\documentclass[11pt]{article}
\usepackage{mathrsfs}
\usepackage{amsmath,amssymb}
\usepackage{amsthm,bm}
\usepackage[colorlinks,linkcolor=red,anchorcolor=blue,citecolor=green]{hyperref}
\usepackage{graphics,graphicx}
\usepackage{color}
\usepackage{enumerate}
\usepackage[title]{appendix}

\usepackage{subfigure}

\usepackage{caption}

\allowdisplaybreaks

\usepackage{tikz}
\usepackage{indentfirst}

\newtheorem{theorem}{Theorem}[section]
\newtheorem{lemma}[theorem]{Lemma}

\newtheorem{proposition}[theorem]{Proposition}
\numberwithin{equation}{section}
\newtheorem{definition}[theorem]{Definition}
\theoremstyle{definition}
\newtheorem{remark}[theorem]{Remark}

\newcommand{\runum}[1]{\romannumeral #1}

\def\R{{\mathbb{R}}}

\def\Z{{\mathbb{Z}}}
\def\N{{\mathbb{N}}}

\def\d{{\mathrm{d}}}

\newcommand{\supp}{\operatorname{supp}}

\title{ \textbf{Notes on eventual continuity and ergodicity for SPDEs}}

 {\author{ \bf{ Ziyu Liu}$^*$ \footnote{$^*$email: liuziyu@math.pku.edu.cn. } \\ 
\small{\em  LMAM, School of Mathematical Science, Peking University, 100871, Beijing, China}
}}

\date{\today}

\begin{document}
	\maketitle

	\makeatother
	
	\begin{abstract} 
        These notes present an alternative approach to the asymptotic stability of stochastic partial differential equations driven by multiplicative noise, applicable to a wide range of dissipative systems. The method builds on general criteria established in \cite{GLLL2024b,L2023},  utilizing the eventual continuity and  generalized coupling techniques.        

        \end{abstract}

    \setcounter{tocdepth}{1}
    \tableofcontents
   
    \section{Introduction}\label{Sec 1}

    In these notes, we present an alternative approach based a weak semigroup regularity to establish the unique ergodicity and asymptotic stability of some dissipative SPDEs with multiplicative noise. Specifically, we adopt the notion of \emph{eventual continuity}, introduced by Gond and Liu \cite{GL2015} (see Definition~\ref{EvC}), or equivalently the asymptotic equicontinuity condition from Jaroszewska \cite{J2013}, as a necessary condition for the asymptotic stability (see Definition \ref{Def 2}). Eventual continuity captures the feature that a uniquely ergodic semigroup may exhibit sensitivity to initial data and is weaker than the e-property (see e.g. \cite{LS2006}). For more general results and discussions on the ergodic properties of eventually continuous Markov semigroups, we refer the interested reader to \cite{GLLL2024b,GLLL2024,LL2024}, as we as the review paper \cite{LL2025}.

    The main tool of our approach is \cite[Theorem 3.16]{GLLL2024b}, or equivalently Theorem \ref{Thm 4}, which establishes the equivalence between asymptotic stability with a unique invariant measure, eventual continuity, and a uniform lower bound condition (\ref{eq 4.1}). Furthermore, we in \cite[Chapter 4]{L2023} extend the concept of eventual continuity, introduce the notion of $d$-eventual continuity (see Definition \ref{Def d-eventual continuity}), and generalize \cite[Theorem 3.16]{GLLL2024b} to derive a weaker criterion for asymptotic stability, as detailed in Theorem~\ref{Thm 5}.

    The core of our approach lies in establishing eventual continuity. We demonstrate this property through the generalized couplings technique introduced in \cite{HM2011} and applied in \cite{NJG2017,KS2018,BKS2020}. It is worth of noting that our application of generalized couplings differs from those in \cite{HM2011, NJG2017, KS2018, BKS2020}. In these works, generalized couplings are used primarily to verify the uniqueness of invariant measures. For instance, in \cite[Corollary 4]{KS2018}, Kulik and Scheutzow apply generalized couplings to provide a sufficient condition for asymptotic stability under the assumption of invariant measures. This approach employees the Birkhoff ergodic theorem to deduce uniqueness, given the existence of invariant measures.

    In contrast, our method does not assume the existence of invariant measures. Instead, we establish both the existence and uniqueness of invariant measures simultaneously, and prove asymptotic stability in the process. The main advantage of the generalized couplings method is its broad applicability and ease of construction for SPDEs. However, since the marginal distribution in our case is not identical to that of the original Markov process, additional work is required to demonstrate eventual continuity through generalized couplings. Compared to the asymptotic strong Feller property or the e-property, eventual continuity is less restrictive, which simplifies certain computations (see more details in the proof). Furthermore, we note that the e-property assumption in \cite[Theorem 3]{KS2018} can be relaxed to eventual continuity with only minor modifications to the proof, as it suffices to focus on the long-time behavior of the Markov process.

    \vspace{0.6em}
    These notes are adapted from the author's doctoral thesis (see \cite[Chapters 4 and 7]{L2023}), defended in May 2023 at Peking University and available in the Peking University Library.
    
    \vspace{0.6em}
    The notes are organized as follows.  Section \ref{Sec 2} introduces some notions, definitions and results regarding ergodicity criteria by means of eventual continuity. In Section \ref{Sec 3}, we establish the asymptotic stability for a class of SPDEs with multiplicative noise,  with the detailed proofs collected in Section \ref{Sec 4}.

    \vspace{0.6em}
    Let $\mathcal{X}$ be a Polish (that is, a complete separable metric) space endowed with the metric $\rho$ and Borel $\sigma$-algebra $\mathcal{B}(\mathcal{X})$.  We use the following notations:\\
 	$\begin{aligned}
        \mathcal{M}(\mathcal{X})&=\text{ the family of finite Borel measures on } \mathcal{X},\\
	\mathcal{P}(\mathcal{X})&=\text{ the family of probability measures on } \mathcal{X},\\
	B_b(\mathcal{X})&=\text{ the space of bounded, Borel real-valued functions defined on }\mathcal{X},\\
	&\quad\;\text { endowed with the supremum norm: } \|f\|_{\infty}=\text{sup}_{x\in \mathcal{X}}\,|f(x)|,f\in B_b(\mathcal{X}),\\
	C_b(\mathcal{X})&=\text{ the subspace of }  B_b(\mathcal{X}) \text{ consisting of bounded continuous functions},\\
	L_b(\mathcal{X})&=\text{ the subspace of } C_b(\mathcal{X}) \text{ consisting of bounded Lipschitz functions},\\
	B(x,r)&=\{y\in \mathcal{X}:\rho(x,y)<r\} \text{ for } x\in \mathcal{X} \text{ and }r>0,\\
	\partial A,\overline{A},\text{Int}_{\mathcal{X}} (A)&=  \text{ the boundary, closure, interior of } A \text{ in } \mathcal{X}, \text{ respectively}, \\
        A^{r}&=\{y\in \mathcal{X}:\rho(x,y)<r,\;x\in A\} \text{ for } A\in \mathcal{B}(\mathcal{X}) \text{ and }r>0,\\
	\supp\mu&=\{x\in \mathcal{X}:\mu(B(x,\epsilon))>0 \text{ for any } \epsilon>0\}, \text{ for }\mu\in\mathcal{M}(\mathcal{X}), \\               &\quad\text{ i.e. the topological  support of the measure } \mu,\\
        \N,\Z,\R,\R_+&= \text{natural numbers, integers, real numbers, nonnegative numbers, respectively}. 
	\end{aligned}$	
   
   \vskip2mm
   For brevity, we use the notation $\langle f,\mu\rangle=\int_{\mathcal{X}}f(x)\mu(\d x)$  for $f\in B_b(\mathcal{X})$ and $\mu\in\mathcal{M}(\mathcal{X})$.

    \section{Preliminaries}\label{Sec 2}

    In this section, we review some necessary concepts related to Markov semigroups and introduce several criteria for unique ergodicity in eventually continuous semigroups. Specifically, Theorem \ref{Thm 4} is derived from \cite{GLLL2024b}, while the new result, Theorem \ref{Thm 5}, based on the $d$-eventual continuity introduced in \cite{GLLL2024b}, is a generalization of Theorem \ref{Thm 4}.

    \begin{definition} \label{Ma-oprt} 
    An operator $P:\mathcal{M}(\mathcal{X})\rightarrow\mathcal{M}(\mathcal{X})$ is a Markov operator on $\mathcal{X}$ if it satisfies the following conditions
    \begin{itemize}
        \item[$(\runum{1})$] {\rm(}Positive linearity{\rm)} $P(\lambda_1\mu_1+\lambda_2\mu_2)=\lambda_1P\mu_1+\lambda_2P\mu_2$ for $\lambda_1,\lambda_2\geq 0$, $\mu_1,\mu_2\in\mathcal{M}(\mathcal{X})$;
        \item[$(\runum{2})$]  {\rm(}Preservation of the norm{\rm)} $P\mu(\mathcal{X})=\mu(\mathcal{X})$ for $\mu\in\mathcal{M}(\mathcal{X})$.
       \end{itemize}
 \end{definition}       
 
\begin{definition}\label{Ma-regul}  
    A Markov operator $P$ is regular if there exists a linear operator $P^*:B_b(\mathcal{X})\rightarrow B_b(\mathcal{X})$ such that
    \begin{equation*}
        \langle f, P\mu\rangle = \langle P^*f, \mu\rangle\quad\forall\,f\in B_b(\mathcal{X}),\;\mu\in\mathcal{M}(\mathcal{X}).
    \end{equation*} 
\end{definition}    
    For the ease of notation, we simply rewrite $P^*$ as $P$.   A {\it Markov semigroup}  $\{P_t\}_{t\geq 0}$ on  $\mathcal{X}$ is a semigroup of Markov operators on $\mathcal{M}(\mathcal{X})$.    
\begin{definition}    \label{Ma-Feller}
 \begin{itemize}     
   \item[$(\runum{1})$]   A Markov operator $P$ is a Markov-Feller operator if it is regular and $P$ leaves $C_b(\mathcal{X})$ invariant, i.e., $P(C_b(\mathcal{X}))\subset C_b(\mathcal{X})$; 
   
   \item[$(\runum{2})$]   A Markov semigroup $\{P_t\}_{t\geq 0}$ is a Markov-Feller semigroup if $P_t$ is a Markov-Feller operator for any $t\geq 0$.  
   
   \item[$(\runum{3})$] A Markov semigroup $\{P_t\}_{t\geq 0}$ is  stochastically continuous if 
			\begin{equation*}
				\lim\limits_{t\rightarrow 0}P_tf(x)=f(x)\quad\forall\,f\in C_b(\mathcal{X}),\;x\in\mathcal{X}.
			\end{equation*}
    \end{itemize}
 
    \end{definition}

   Recall that $\mu\in\mathcal{P}(\mathcal{X})$ is \emph{invariant} for the semigroup $\{P_t\}_{t\geq 0}$ if $P_t\mu=\mu$ for any $t\geq 0$.

    \begin{definition}\label{EvC}
		A Markov semigroup 
		$\{P_t\}_{t\geq 0}$ is eventually continuous at $z\in \mathcal{X},$ if for any Lipschitz bounded function $f,$
		\begin{equation}\label{def1}
		\limsup\limits_{x\rightarrow z}\limsup\limits_{t\rightarrow\infty}|P_tf(x)-P_tf(z)|=0,
		\end{equation}
		that is, $\forall\,\epsilon>0$, $\exists\,\delta>0$ such that $\forall\,x\in B(z,\delta)$, there exists some $T_x\geq 0$ satisfying
		\begin{equation*}
			|P_tf(x)-P_tf(z)|\leq \epsilon\quad\forall\, t\geq T_x.
		\end{equation*}		
    \end{definition}

    \begin{definition}\label{Def 2}
   
		$\{P_t\}_{t\geq 0}$ is asymptotically stable if there exists a unique invariant measure $\mu_*\in\mathcal{P}(\mathcal{X})$ such that $\{P_t\mu\}_{t\geq 0} $ weakly converges to $\mu_*$  as $t\rightarrow\infty$ for any $\mu\in\mathcal{P}(\mathcal{X})$, i.e., $\langle f,P_t\mu\rangle\rightarrow\langle f,\mu_*\rangle$ as $t\rightarrow\infty$ for any $f\in C_b(\mathcal{X})$.
    \end{definition}

    \begin{theorem}{\rm (\cite[Theorem 3.16]{GLLL2024b}, \cite[Theorem 4.1]{L2023})}\label{Thm 4} 
		The following three statements are equivalent:
		\begin{itemize}
		    \item[$(\runum{1})$] $\{P_t\}_{t\ge 0}$ is asymptotically stable with unique invariant measure $\mu$.
		    \item[$(\runum{2})$] $\{P_t\}_{t\geq 0}$ is eventually continuous on $\mathcal{X}$, and there exists $z\in\mathcal{X}$ such that for any $\epsilon>0$,  
		  	\begin{equation}\label{eq 4.1}\tag{$\mathcal{C}$}
		  	   \inf\limits_{x\in \mathcal{X}}\liminf\limits_{t\rightarrow\infty}P_t(x,B(z,\epsilon))>0.
		\end{equation}

              \item[$(\runum{3})$]  There exists $z\in\mathcal{X}$ such that $\{P_t\}_{t\geq 0}$ is eventually continuous at $z$ and for any $\epsilon>0$, \eqref{eq 4.1} holds. 
	\end{itemize} 
		
	\end{theorem}

    \vspace{0.3em}

    In Section \ref{Sec 3}, we present an approach to verify the asymptotic stability of a broad class of SPDEs using Theorem \ref{Thm 4}. However, some of these SPDE models are not Feller processes, and in some cases, verifying the eventual continuity is not straightforward. To address these issues, we introduce a generalized notion of eventual continuity, called \emph{$d$-eventual continuity}, and derive a more general criterion to handle such cases.

    Recall that $(\mathcal{X},\rho)$ is a Polish space. Suppose that $\mathcal{X}$ has an alternative distance $d$, which is continuous with respect to $\rho$. Note that $d$ is not necessarily complete. When $d \neq \rho$, we denote by $\overline{\mathcal{X}}^d$ the completion of $\mathcal{X}$ with respect to $d$ and regard $\mathcal{X}$ as a subset of $\overline{\mathcal{X}}^d$. Let $\{P_t\}_{t \geq 0}$ be a Markov semigroup on $\mathcal{X}$. Following \cite{KS2018}, we say that $\{P_t\}_{t \geq 0}$ is $d$-Feller if, for each bounded and $d$-continuous function $f$, $P_t f$ is $d$-continuous for all $t \geq 0$. Building on the concept of $d$-equicontinuity from \cite{KS2018}, we introduce the notion of $d$-eventual continuity as follows.
	
	\begin{definition}{\rm (\cite[Definition 4.3]{L2023})}\label{Def d-eventual continuity}
		A Markov semigroup $\{P_t\}_{t\geq 0}$ is said to be $d$-eventually continuous at $z\in\mathcal{X}$, if for any $f\in C_b(\mathcal{X})\cap L_b(\mathcal{X},d),$
		\begin{equation*}
		\lim\limits_{x\rightarrow z}\limsup\limits_{t \to \infty}|P_tf(x)-P_tf(z)|=0,
		\end{equation*}
		where $L_b(\mathcal{X},d)$ denotes the collection of bounded $d$-Lipschitz continuous functions.
	\end{definition}

		Using these notions, we have the following theorem.
	
	\begin{theorem}{\rm (\cite[Theorem 4.2]{L2023})}\label{Thm 5} 
		Assume that Markov semigroup $\{P_t\}_{t\geq 0}$ is Feller or $d$-Feller,
		and that $\{P_t\}_{t\geq 0}$ is $d$-eventually continuous on $\mathcal{X}$. Assume that there exists some $z\in\mathcal{X}$ such that for any $\epsilon>0$,         
        \begin{equation*}
		  	   \inf\limits_{x\in \mathcal{X}}\liminf\limits_{t\rightarrow\infty}P_t(x,B_d(z,\epsilon))>0,
		\end{equation*}
        where $B_d(z,\epsilon):=\{x\in\mathcal{X}:d(z,x)<\epsilon\}$. Then $\{P_t\}_{t\geq 0}$ is asymptotically stable in the $d$-topology, that is, there exists unique invariant measure $\mu\in\mathcal{M}_1(\mathcal{X})$ such that, for all $x\in\mathcal{X}$, $P_t(x,\cdot)$ weakly converges to $\mu$ in the $d$-topology as $t\rightarrow\infty$.

	\end{theorem}

        \begin{remark}
            The notion of $d$-eventual continuity also applies to discrete-time Markov semigroups, and Theorem \ref{Thm 5} can be readily extended to this setting.
        \end{remark}

        To prove Theorem \ref{Thm 5}, it is sufficient to observe two key facts. First, since $d$ is continuous with respect to $\rho$, any compact set $K$ remains compact in the space $(\overline{\mathcal{X}}^d,d)$. Second, the Feller or $d$-Feller property is only used to establish the existence of invariant measures, as seen in the proof of \cite[Proposition 3.8]{GLLL2024b}, and both types of Feller properties are adequate for our purposes.

        \begin{proof}[Proof of Theorem \ref{Thm 5}] 
        
        The proof consists of two parts.

        \noindent $\mathbf{Part\;1:\;existence\;of\;invariant\;measures}.$
        
		We extend $\{P_t(x,\cdot):x\in\mathcal{X}\}_{t\geq 0}$ to $\{\overline{P}_t(x,\cdot):x\in\overline{\mathcal{X}}^d\}_{t\geq 0}$ by
		\begin{equation*}
		\overline{P}_t(x,\cdot)=\begin{cases}
		P_t(x,\cdot),\quad &x\in\mathcal{X};\\
		P_t(z,\cdot),\quad &x\notin\mathcal{X}.
		\end{cases}
		\end{equation*}
		Then  $\{\overline{P}_t\}_{t\geq 0}$ is a Markov semigroup on $(\overline{\mathcal{X}}^d,d)$ and satisfies that		
		\begin{equation*}
		\inf\limits_{x\in \overline{\mathcal{X}}^d}\liminf\limits_{t\rightarrow\infty}\overline{P}_t(x,B_d(z,\epsilon))>0,
		\end{equation*}
	   Furthermore, $\{\overline{P}_t\}_{t\geq 0}$ is eventually continuous at $z$. Fix any $\bar{f}\in L_b(\overline{\mathcal{X}^d})$, and let $f:=\bar{f}|_{\mathcal{X}}.$ Then
	
		\begin{equation*}
			\overline{P}_tf(x)-\overline{P}_tf(z)=\begin{cases}
			P_tf(x)-P_tf(z),\quad &x\in\mathcal{X};\\
			0,\quad &x\notin\mathcal{X},
			\end{cases}
		\end{equation*}
		which implies the eventual continuity of $\{\overline{P}_t\}_{t\geq 0}$ immediately by assumptions.

		Denote $\overline{Q}_t(z,\cdot)=t^{-1}\int_{0}^{t}\overline{P}_s(z,\cdot)ds$. Note that $\overline{Q}_t(z,\cdot)=Q_t(z,\cdot),\;t>0$.
		We claim that $\{\overline{Q}_t(z,\cdot)\}_{t\geq 0}$ is tight in $(\overline{\mathcal{X}}^d,d)$.  Assume, contrary to our claim that $\{\overline{Q}_t(z,\cdot)\}_{t\geq 0}$ is not tight, then by the same arguments as in \cite[Proposition 3.8]{GLLL2024b}, there exists some $\epsilon>0,$ a sequence of compact sets $\widetilde{K}_i\subset \overline{\mathcal{X}}^d,\;i \geq 1$, and an increasing sequence of reals $\{t_i\}_{i\geq 1},t_i\rightarrow\infty,$ such that
		\begin{align*}
		\overline{Q}_{t_i}(z,\widetilde{K}_i)&\geq \epsilon,\quad\forall i\geq 1,\\
		\widetilde{K}_{i,d}^{\epsilon/2}\cap \widetilde{K}_{j,d}^{\epsilon/2}&=\emptyset,\quad \forall i,j\geq 1,\;i\neq j,
		\end{align*}
		where $\widetilde{K}_{i,d}^{\epsilon/2}:=\{x\in\overline{\mathcal{X}^d}:d(x,y)<\epsilon/2,\;y\in \widetilde{K}_i\}$. \par 		
		Now let $K_i:=\widetilde{K}_i\cap \mathcal{X}$ and	$K_{i,d}^{\epsilon/4}:=\{x\in\mathcal{X}:d(x,y)<\epsilon/4,\;y\in K_i\}$, then $K_i$ is compact and $K_{i,d}^{\epsilon/4}\subset \widetilde{K}_{i,d}^{\epsilon/2}$ for $i\geq 1$. In addition,
		\begin{equation*}
		Q_{t_i}(z,K_i)=\overline{Q}_{t_i}(z,\widetilde{K}_i)\geq \epsilon,\quad\forall i\geq 1.
		\end{equation*}\par 
		For each $i\geq 1$, define $f_i:\mathcal{X}\rightarrow \mathbb{R}$ by
		\begin{equation*}
		f_i(y)=d(y,(K_{i,d}^{\epsilon/4})^c)/(d(y,(K_{i,d}^{\epsilon/4})^c)+d(y,K_i)),
		\end{equation*}
		which fulfills that $f_i\in C_b(\mathcal{X})\cap L_b(\mathcal{X},d)$, and $\mathbf{1}_{K_i}\leq f_i\leq\mathbf{1}_{K_i^{\epsilon/4}}.$ Thus $\{P_tf_i\}_{t\geq 0}\subset C_b(\mathcal{X})$ whether $\{P_t\}_{t\geq 0}$ is Feller or $d$-Feller. Then the rest of the proof remains the same as that of \cite[Proposition 3.8]{GLLL2024b}.

	Consequently, we show that  $\{\overline{Q}_t(z,\cdot)\}_{t\geq 0}$ is tight in $(\overline{\mathcal{X}}^d,d)$, and it guarantees some invariant measure $\bar{\mu}$ for $\{\overline{P}_t\}_{t\geq 0}$. Clearly,  $\bar{\mu}$ is supported in $\mathcal{X}$ by the construction of $\{\overline{P}_t\}_{t\geq 0}.$ Consequently, let $\mu:=\bar{\mu}|_{\mathcal{X}}$, i.e.,
		\begin{equation*}
			\mu(A)=\bar{\mu}(A)\quad\text{ for }A\in\mathcal{B}(\mathcal{X}),
		\end{equation*}
		then $\mu$ is an invariant measure for $\{P_t\}_{t\geq 0}$.

	\vspace{0.3em}

	\noindent $\mathbf{Part\;2:\;asymptotic\;stability}.$	\par  
	It suffices to show that
	\begin{equation}\label{eq Thm 4.2 1}
	\lim\limits_{t\rightarrow\infty}|P_tf(x)-P_tf(y)|=0
	\end{equation}
	holds for all $x,y\in \mathcal{X}$ and $f\in L_b(\mathcal{X},d).$\par 
	Assume, contrary to our claim, that (\ref{eq Thm 4.2 1}) fails for some  $f\in L_b(\mathcal{X},d)$ and points $x_1,x_2\in \mathcal{X}$ such that
	\begin{equation*}
	\limsup\limits_{t\rightarrow\infty}|P_tf(x_1)-P_tf(x_2)|>0.
	\end{equation*}
	Choose $\epsilon>0$ such that
	\begin{equation*}
	\limsup\limits_{t\rightarrow\infty}|P_tf(x_1)-P_tf(x_2)|\geq 3\epsilon.
	\end{equation*}
	$\{P_t\}_{t\geq 0}$ is $d$-eventually continuous at $z\in \mathcal{X},$ thus we can choose $\delta>0$ such that
	\begin{equation*}
	\limsup\limits_{t\rightarrow\infty}|P_tf(x)-P_tf(z)|<\epsilon/2\quad\text{for all }x\in B_d(z,\delta). 
	\end{equation*}\par 
	Condition (\ref{eq 4.1}) gives some $\alpha>0$ such that
	\begin{equation*}
	\liminf\limits_{t\rightarrow\infty}P_t(x,B_d(z,\delta))>\alpha\quad\text{for all }x\in \mathcal{X}.
	\end{equation*}
	Then Fatou's lemma gives, for any $\nu\in\mathcal{M}_1(\mathcal{X})$ with $\text{supp}_d\;\nu\subset B_d(z,\delta),$
	\begin{equation*}
	\liminf\limits_{t\rightarrow\infty}P_t\nu(B_d(z,\delta))\geq\int_{\mathcal{X}}\liminf\limits_{t\rightarrow\infty}P_t(y,B_d(z,\delta))\nu(dy)>\alpha,
	\end{equation*}
	where $\text{supp}_d\;\nu:=\{x\in\mathcal{X}:\nu(B_d(z,\gamma))>0\text{ for all } \gamma>0\}$.\par 
	Let $k\geq 1$ be such that $2(1-\alpha)^k\|f\|_{\infty}<\epsilon.$ By induction we are going to define four sequences of measures $\{\nu_i^{x_1}\}_{i=1}^k,\{\mu_i^{x_1}\}_{i=1}^k,\{\nu_i^{x_2}\}_{i=1}^k,\{\mu_i^{x_2}\}_{i=1}^k,$ and a sequence of positive numbers $\{t_{i}\}_{i=1}^k$ in the following way: let
	$t_1>0$ be such that
	\begin{center}
		$P_{t_{1}}(x_j,B_d(z,\delta))>\alpha,\quad j=1,2.$ 
	\end{center}\par
	Set 
	\begin{center}
		$\nu_1^{x_j}(\cdot) = \dfrac{P_{t_1}\delta_{x_j}(\cdot \cap B_d(z,\delta))}{P_{t_1}\delta_{x_j}(B_d(z,\delta))},\quad$
		$\mu_1^{x_j}(\cdot) = \dfrac{1}{1-\alpha}(P_{t_1}\delta_{x_j}(\cdot)-\alpha\nu_1^{x_j}(\cdot)),\quad j=1,2.$
	\end{center}\par
	Assume that we have done it for $i = 1,\dots , l,$ for some $l < k.$ Now let $t_{l+1}$ be such that
	\begin{equation*}
	P_{t_{1+1}}\mu_l^{x_j}(B_d(z,\delta))>\alpha,\quad j=1,2.
	\end{equation*}\par
	Set 
	\begin{center}
		$\nu_{l+1}^{x_j}(\cdot) = \dfrac{P_{t_{l+1}}\mu_l^{x_j}(\cdot \cap B_d(z,\delta))}{	P_{t_{1+1}}\mu_l^{x_j}(B_d(z,\delta))},\quad$
		$\mu_1^{x_j}(\cdot) = \dfrac{1}{1-\alpha}(P_{t_{l+1}}\mu_l^{x_j}(\cdot)-\alpha\nu_{l+1}^{x_j}(\cdot)),\quad j=1,2.$
	\end{center}\par
	\begin{equation*}
	\begin{aligned}
	P_{t_1+\dots+t_k}\delta_{x_j}(\cdot)&=\alpha P_{t_2+\dots+t_k}\nu_1^{x_j}(\cdot)+\alpha(1-\alpha) P_{t_3+\dots+t_k}\nu_2^{x_j}(\cdot)+ \dots +\\
	&+\alpha(1-\alpha)^{k-1}\nu_k^{x_j}(\cdot)+(1-\alpha)^k \mu_k^{x_j}(\cdot),
	\end{aligned}
	\end{equation*}
	where $\text{supp}_d\;{{\nu}_i^{x_j}}\subset B_d(z,\delta),\;j=1,2.$ \par 
	Thus
	\begin{equation*}
	\begin{aligned}
	\limsup\limits_{t\rightarrow\infty}|\langle P_{t}f, \nu_i^{x_1} \rangle-\langle P_{t}f, \nu_i^{x_2} \rangle|&=	\limsup\limits_{t\rightarrow\infty}|\langle P_{t}f-P_{t}f(z), \nu_i^{x_1} \rangle-\langle P_{t}f-P_{t}f(z), \nu_i^{x_2} \rangle|\\
	&\leq\epsilon/2+\epsilon/2=\epsilon.
	\end{aligned}
	\end{equation*}
	The Fatou's lemma gives 
	\begin{equation*}
	\begin{aligned}
	3\epsilon&\leq\limsup\limits_{t\rightarrow\infty}|P_tf(x_1)-P_tf(x_2)|\\
	&=\limsup\limits_{t \to \infty}|\langle f,P_{t} \delta_{x_1}\rangle-\langle f,P_{t} \delta_{x_2}\rangle| \\
	&\leq \alpha\limsup\limits_{t \to \infty}|\langle P_{t}f, \nu_1^{x_1} \rangle-\langle P_{t}f, \nu_1^{x_2} \rangle|+\cdots+\alpha(1-\alpha)^k\limsup\limits_{t \to \infty}|\langle P_{t}f, \nu_k^{x_1}\rangle-\langle P_{t}f, \nu_1^{x_2} \rangle|+\epsilon\\
	&\leq (\alpha+\cdots+\alpha(1-\alpha)^k)\epsilon+\epsilon\\
	&< 2\epsilon,
	\end{aligned}
	\end{equation*}
	which is impossible. This completes the proof.

    \end{proof}

     \section{Ergodicity for some stochastic partial differential equations }\label{Sec 3}

    In this section, we present an approach for establishing the unique ergodicity of certain SPDEs with multiplicative noise. The primary tool we use is Theorems \ref{Thm 4} and \ref{Thm 5}. Specifically, to prove asymptotic stability, it suffices to establish the \emph{eventual continuity} and the \emph{lower bound condition} \eqref{eq 4.1} at a single point for the Markov-Feller semigroup $\{P_t\}_{t\geq 0}$ associated with the given equation. Our approach simultaneously provides the existence and uniqueness of the invariant measure, along with weak convergence.

    As mentioned earlier, the eventual continuity is established using the generalized coupling method. One advantage of this method is its technical convenience, as it allows for the construction of suitable processes in the context of multiplicative perturbations. For deriving the lower bound condition, we utilize a Lyapunov structure combined with a form of irreducibility. More specifically, we show that relation \eqref{eq 4.1} can be guaranteed as follows.

    \begin{proposition}{\rm (\cite[Proposition 7.1]{L2023})}\label{prop lbc}
   The lower bound condition \eqref{eq 4.1} of Theorem \ref{Thm 4} $(\runum{2})$ is satisfied for $\{P_t\}_{t\geq 0}$ if the following two conditions hold:    
        \begin{itemize}
            \item[$(\runum{1})$] {\rm(}Lyapunov structure{\rm)} There exist measurable functions $V\colon\mathcal{X}\rightarrow\R_+$,  $h\colon\R_+\rightarrow\R_+$ satisfying  $\lim_{t\rightarrow\infty}h(t)=0$,  and a constant $C>0$ such that
            \begin{equation}\label{eq lyapunov}
                P_tV(x)\leq h(t)V(x)+C\quad \forall\,x\in\mathcal{X},\;t\geq 0.
            \end{equation}
             \item[$(\runum{2})$] {\rm(}Uniform irreducibility{\rm)}  There exists  $z\in\mathcal{X}$ such that for any $\epsilon,R>0$, there is $T=T(\epsilon,R)>0$ satisfying
		\begin{equation}\label{eq irreducibility}
		\inf\limits_{x\in \{V\leq R\}}P_T(x,B(z,\epsilon))>0.
		\end{equation} 
        \end{itemize}
          
    \end{proposition}

    \begin{remark}
        In the original version of \cite[Proposition 7.1]{L2023}, we use some energy estimates to guarantee the growth at the infinity. Namely, we propose that the lower bound condition \eqref{eq 4.1} is satisfied if the uniform irreducibility \eqref{eq irreducibility} holds, and there exists $z\in\mathcal{X}$, $k\geq 1$, $C>0$, such that for each $x\in\mathcal{X}$, there exists a nonincreasing function $r_x:\mathbb{R}_+\rightarrow\mathbb{R}_+$ with $\lim\limits_{t\rightarrow\infty}r_x(t)=0$, and that
    	\begin{equation*}
    	\mathbb{E}\,\rho(\mathbf{u}^x(t),z)^k\leq r_x(t)+C,
	    \end{equation*}
        where $\{\boldsymbol{u}^x(t)\}_{t\geq 0, x\in\mathcal{X}}$ denotes the Markov process associated with the semigroup $\{P_t\}_{t\geq 0}$.

        Thanks to insightful discussions with Yuan Liu, in the present notes, we adopt a more general framework by using Lyapunov functionals \eqref{eq lyapunov} to replace the energy estimates condition.
    \end{remark}

    \begin{proof}[Proof of Proposition \ref{prop lbc}.]
    By \eqref{eq lyapunov}, for any $R>0$ and $x\in\mathcal{X}$, it follows that
	\begin{equation*}
	\begin{aligned}	
	P_t(x,\{V\leq R\})&\geq 1-R^{-1}P_tV(x)\geq 1-R^{-1}(h(t)V(x)+C).
	\end{aligned}
	\end{equation*} 
	Given that $\lim_{t\rightarrow\infty}h(t)=0$, there exists $R>0$ sufficiently large such that
	\begin{equation}\label{Equation lbc1}
		\liminf\limits_{t\rightarrow\infty}P_t(x,\{V\leq R\})\geq \tfrac{1}{2}\quad\forall\,x\in \mathcal{X}.
	\end{equation} 
	Meanwhile, using the uniform irreducibility \eqref{eq irreducibility}, for any $\epsilon>0$, there exists $T,p>0$ satisfying
	\begin{equation}\label{Equation lbc2}
		P_T(x,B(0,\epsilon))\geq p>0\,\quad\forall\,x\in \{V\leq R\}. 
	\end{equation}

    Summarizing \eqref{Equation lbc1} and  \eqref{Equation lbc2}, for any $t>0$, we derive that 
	\begin{equation*}
	\begin{aligned}
	P_{t+T}(x,B(0,\epsilon))&=\int_{\mathcal{X}}P_t(x,\d y)P_{T}(y,B(0,\epsilon)))\d y\\
    &\geq\int_{\{V\leq R\}}P_t(x,\d y)P_{T}(y,B(0,\epsilon)))\d y\\
	&\geq P_t(x,\{V\leq R\})\cdot \inf\limits_{x\in \{V\leq R\}}P_T(x,B(0,\epsilon)),
	\end{aligned}
	\end{equation*}
    which implies that     
	\begin{equation*}			\liminf\limits_{t\rightarrow\infty}P_t(x,B(0,\epsilon))\geq \tfrac{p}{2}>0\quad \forall\,x\in\mathcal{X}.
	\end{equation*}
	\end{proof}
    
	In summary, our approach is divided into the verification of following four properties: 
    \begin{itemize}
        \item [$(\runum{1})$]  Feller property; 
        \item [$(\runum{2})$] Eventual continuity;
        \item [$(\runum{3})$] Lyapunov structure;
        \item [$(\runum{4})$] Uniform irreducibility.
    \end{itemize}

    Usually, the Feller property and Lyapunov structure would follow from the regularity and dissipation of the system.  The way we show the eventual continuity is through the generalized couplings approach introduced in \cite{HM2011} and applied in \cite{NJG2017,KS2018,BKS2020}.  As for the uniform irreducibility, when dealing with SPDEs driven by additive noise, this problem can typically be solved by control methods, for example, see in \cite{EM2001}. Furthermore, when the noise is nondegenerate multiplicative, the uniform irreducibility can be argued by control methods, for example, see in \cite{Z2009,HM2011}.

    \subsection{The 2D stochastic Navier--Stokes equation}\label{Section NS}
 	
	To clarify the basic logic of our method, we begin with a well-known example, the 2D stochastic Navier--Stokes equation with multiplicative noise. Despite there are fruitful results about the the unique ergodicity of the 2D stochastic Navier--Stokes equation, for example, see  \cite{FM1995,M1999,EM2001,EMS2001,BKL2001,KS2001,KPS2002,M2002,HM2006,MP2006,O2008,DP2018}, it is quite helpful to illustrate the typical procedures, and capture the key points of the eventual continuity approach.

	Recall	the 2D stochastic Navier--Stokes equation with multiplicative noise posed on a bounded domain $D\subset\mathbb{R}^2$ with a smooth boundary $\partial D:$
	\begin{equation}\label{Equation NS}
	\left\{\begin{aligned}
	&d\mathbf{u}+\mathbf{u}\cdot\nabla\mathbf{u}dt=(\nu\Delta\mathbf{u}-\nabla p)dt+\sum_{k=1}^{m}\sigma_k(\mathbf{u})dW_t^k,\\
	&\mathbf{u}_0=x,\quad\nabla\cdot\mathbf{u}=0,\quad
	\mathbf{u}|_{\partial D}=0,
	\end{aligned}\right.
	\end{equation}
	where $\mathbf{u}=(u_1,u_2)$ is the unknown velocity field,  $p$ is the unknown pressure, $m\in\mathbb{N}$, $W=(W^1,\dots,W^m)$ is a standard $m$-dimensional Brownian motion, $\sigma_1,\dots,\sigma_m:H\rightarrow H,$ are measurable mappings, $\nu>0.$ \par

	Recall the standard notation. Consider (\ref{Equation NS}) on the phase space
	\begin{equation*}
	H:=\{\mathbf{u}\in L^2(D)^2:\nabla\cdot\mathbf{u}=0,\;\mathbf{u}\cdot\mathbf{n}=0\},
	\end{equation*}
	where $\mathbf{n}$ is the outward normal to $\partial D$. Denote $P_L$ as the orthogonal projection of $L^2(D)^2$ onto $H$. The space of vector fields whose gradients are integrable in $L^2(D)^2$ is also relevant
	and we define 
	\begin{equation*}
	V := \{\mathbf{u}\in H^1(D)^2 :\nabla\cdot\mathbf{u}=0,\; \mathbf{u}|_{\partial D}=0\}.
	\end{equation*}
	We denote the norms associated to $H$ and $V$ respectively as $|\cdot|$ and $\|\cdot\|.$ \par 
	
	The Stokes operator is defined as $ Au = -P_L \Delta u $, for any vector field $\mathbf{u}\in V\cap H^2(D)^2.$ Since $A$ is self-adjoint with a compact inverse, we infer that $A$ admits an increasing sequence of eigenvalues $\lambda_k \sim k$ diverging to infinity with the corresponding eigenvectors $e_k$ forming a complete orthonormal basis for $H$. We denote by $P_N$ and $Q_N$ the projection onto $H_N = \text{span\;}\{e_k : k = 1,\dots, N \}$ and its orthogonal complement, respectively. Recall the generalized Poincar\'e inequalities
	
	\begin{equation}\label{Equation NS3}
	\|P_N\mathbf{u}\|^2\leq\lambda_N|P_N\mathbf{u}|^2,\quad|Q_N\mathbf{u}|^2\leq\lambda_N^{-1}\|Q_N\mathbf{u}\|^2
	\end{equation}
	hold for all sufficiently smooth $\mathbf{u}$ and any $N\geq 1.$\par 
	We make the following Hypotheses: 
	\begin{itemize}
		\item[$\mathbf{H_1.}$] 	The function $\sigma=(\sigma_1,\dots,\sigma_m)$ is bounded and Lipschitz, i.e., there exist constants $B_0,L>$ such that
		\begin{equation*}
		\begin{aligned}
		&|\sigma(\mathbf{u})|^2=\sum_{k=1}^{m}|\sigma_k(\mathbf{u})|^2\leq B_0,\text{ for all }\mathbf{u}\in H;\\
		&|\sigma(\mathbf{u})-\sigma(\mathbf{v})|^2=\sum_{k=1}^{m}|\sigma_k(\mathbf{u})-\sigma_k(\mathbf{v})|^2\leq L|\mathbf{u}-\mathbf{v}|^2,\text{ for all }\mathbf{u},\mathbf{v}\in H.
		\end{aligned}
		\end{equation*}
		\item[$\mathbf{H_2.}$] There exists $N\in\mathbb{N}$ such that
		\begin{equation*}
		P_NH\subset \text{Range}\;(\sigma(\mathbf{u})),\text{ for all }\mathbf{u}\in H.
		\end{equation*}\par 
		Moreover, the corresponding pseudo-inverse operators $\sigma_k^{-1}:P_NH\rightarrow H$ are uniformly bounded, i.e., there exists a constant $C_0$ such that
		\begin{equation*}
		|\sigma(\mathbf{u})^{-1}(P_N\mathbf{w})|\leq C_0|P_N\mathbf{w}|,\text{ for all }\mathbf{u},\mathbf{w}\in H.
		\end{equation*}\par 
		\item[$\mathbf{H_3.}$] 	Finally, we assume that $N$ is sufficiently large such that
		\begin{equation}\label{Equation NS14.3}
		\lambda_N>\frac{L}{\nu}+\frac{C_D^2}{\nu^3}B_0.
		\end{equation}\par 
	\end{itemize}
	
	The existence of a unique strong solution to equation (\ref{Equation NS}) is standard, see e.g. \cite{O2008}. Furthermore, as the energy estimates  and Feller property can be obtained directly, we prove these two parts first. \par 
	
	\paragraph{Energy estimates} 
	\begin{proposition}\label{NS Prop EE} Assume $\mathbf{H_1}$. Then
		\begin{equation}\label{Equation NS21}
		\mathbb{E}|\mathbf{u}|^2\leq e^{-2\nu t}|x|^2+\frac{B_0}{2\nu}.
		\end{equation}
	\end{proposition}
	\begin{proof}[Proof.] 
        Applying It\^o lemma to \ref{Equation NS}, we have 
        \begin{equation*}
            d|\mathbf{u}|^2+2\nu\|\mathbf{u}\|^2de=|\sigma(\mathbf{u})|^2dt+2\langle\sigma(\mathbf{u},\mathbf{u}\rangle dW.
        \end{equation*}
		Let $K>0,$ we introduce the stopping time
		\begin{equation*}
		\tau:=\inf\{t\geq 0:|\mathbf{u}|^2\geq K^2\}.
		\end{equation*}\par 
		Denoting by $M_t=\int_{0}^{t}2\langle\sigma(\mathbf{u},\mathbf{u}\rangle dW_s$ the local martingale term, we define the stopped martingale $M_t^\tau$ by
		\begin{equation*}
		M_t^\tau=\int_{0}^{t}2\langle\sigma(\mathbf{u}_{s\wedge\tau}),\mathbf{u}_{s\wedge\tau}\rangle dW_s,
		\end{equation*}
		then 
		\begin{equation*}
		\langle M^\tau\rangle_t=\int_{0}^{t}4|\langle\sigma(\mathbf{u}_{s\wedge\tau}),\mathbf{u}_{s\wedge\tau}\rangle|^2ds\leq 4B_0K^2t<\infty.
		\end{equation*}
		Hence we know that $\mathbb{E}M_t^\tau=0.$ And because $t\wedge\tau$ is a bounded stopping time, the Optional Stopping Time Lemma says that $\mathbb{E}M_{t\wedge\tau}^\tau=0$. As $\mathbb{E}M_{t\wedge\tau}=\mathbb{E}M_{t\wedge\tau}^\tau$, we have
		\begin{equation*}
		\mathbb{E}|\mathbf{u}_{t\wedge\tau}|^2+2\nu \mathbb{E}\int_{0}^{t\wedge\tau}\|\mathbf{u}\|^2ds=|x|^2+\mathbb{E}\int_{0}^{t\wedge\tau}|\sigma(\mathbf{u})|^2ds.
		\end{equation*}
		Noting that $|\mathbf{u}_t|$ is continuous in time, and $\tau\rightarrow\infty$ as $K\rightarrow\infty,$ hence $t\wedge\tau\rightarrow t.$ Thus we obtain
		\begin{equation*}
		\mathbb{E}|\mathbf{u}|^2+2\nu \mathbb{E}\int_{0}^{t}\|\mathbf{u}\|^2ds=|x|^2+\mathbb{E}\int_{0}^{t}|\sigma(\mathbf{u})|^2ds.
		\end{equation*}
		Further, we have
		\begin{equation*}
		\mathbb{E}|\mathbf{u}|^2+2\nu \mathbb{E}\int_{0}^{t}|\mathbf{u}|^2ds\leq|x|^2+B_0t.
		\end{equation*}
		By Gronwall inequality, we have		
		\begin{equation*}
		\mathbb{E}|\mathbf{u}|^2\leq e^{-2\nu t}|x|^2+\frac{B_0}{2\nu}(1-e^{-2\nu t}),
		\end{equation*}
		completing the proof.
	\end{proof}

	\paragraph{Feller property} Applying similar arguments as in the energy estimates, $\mathbf{u}$ can be further showed to be a Feller Markov process with state space $H$, that is, the transition functions $P_t (x,A) := \mathbb{P}(\mathbf{u}(t, x) \in A),$ are well defined for any $x\in H,\;t\geq 0$, and any Borel subset $A$ of $H,$ and define an associated Markov-Feller semigroup $\{P_t\}_{t\geq 0}$ on $C_b(H)$. The proof is postponed in the Appendix.
	\begin{proposition}\label{NS Prop Feller} Assume that $\mathbf{H_1}$ holds. Then $\{P_t\}_{t\geq 0}$ is Feller.
	\end{proposition}

	\par 
	Next we use the generalized couplings arguments to show the eventual continuity.

	\paragraph{Eventual continuity}
	 It was shown in \cite{O2008} that for any initial condition $x\in H$ this equation has a unique strong solution, which in the case of ambiguity will be denoted later by $\mathbf{u}^x$. Fix any $x,y\in H$, take $\mathbf{u}^x$ solving (\ref{Equation NS}) with initial datum $x,$ and define $\tilde{\mathbf{u}}^y$ as the solution to the following stochastic 2D Navier--Stokes equation (\ref{Equation NS}) with the initial condition $y$ and the additional
	control term: 
	\begin{equation}\label{Equation NS6}
	\left\{\begin{aligned}
	&d\tilde{\mathbf{u}}^y+\tilde{\mathbf{u}}^y\cdot\nabla\tilde{\mathbf{u}}^ydt=(\nu\Delta\tilde{\mathbf{u}}^y-\nabla\tilde{p})dt\\
	&\qquad\qquad\qquad\quad+\frac{\nu\lambda_N}{2} P_N(\mathbf{u}^x-\tilde{\mathbf{u}}^y)dt+\sigma(\tilde{\mathbf{u}}^y)dW_t,\\
	&\tilde{\mathbf{u}}^y_0=y,\quad\nabla\cdot\tilde{\mathbf{u}}^y=0,\quad  	\tilde{\mathbf{u}}^y|_{\partial D}=0,
	\end{aligned}\right.
	\end{equation}
	where $N\geq 1$ which we will specify below. \par 
	Let $\mathbf{v}:=\mathbf{u}^x-\tilde{\mathbf{u}}^y$ and $q:=p-\tilde{p}.$ We have that
	\begin{equation}\label{Equation NS7}
	\left\{\begin{aligned}
	&d\mathbf{v}=\nu\Delta\mathbf{v}dt-\frac{\nu\lambda_N}{2} P_N\mathbf{v}dt-\nabla qdt-\mathbf{v}\cdot\nabla\mathbf{u}^xdt-\tilde{\mathbf{u}}^y\cdot\nabla\mathbf{v}dt+(\sigma(\mathbf{u}^x)-\sigma(\tilde{\mathbf{u}}^y))dW_t,\\
	&\mathbf{v}_0=x-y,\quad\nabla\cdot\mathbf{v}=0,\quad\mathbf{v}|_{\partial D}=0.
	\end{aligned}\right.
	\end{equation}\par 
	We proceed to establish some estimates on $\mathbf{v}$ to imply the eventual continuity. Applying the It\^o lemma to (\ref{Equation NS7}), we find that
	\begin{equation}\label{Equation NS8.1}
	\begin{aligned}
d|\mathbf{v}|^2+2\nu\|\mathbf{v}\|^2dt+\nu\lambda_N|P_N\mathbf{v}|^2dt&=-2\langle\mathbf{v}\cdot\nabla\mathbf{u}^x,\mathbf{v}\rangle dt+|\sigma(\mathbf{u}^x)-\sigma(\tilde{\mathbf{u}}^y)|^2dt\\
 &\quad+\langle\sigma(\mathbf{u}^x)-\sigma(\tilde{\mathbf{u}}^y),\mathbf{v}\rangle dW_t.
	\end{aligned}
	\end{equation}
	We have a generic bound
	\begin{equation}\label{Equation NS8}
	|\langle\mathbf{v}\cdot\nabla\mathbf{u},\mathbf{v}\rangle |\leq C_D|\mathbf{v}|\,\|\mathbf{v}\|\,\|\mathbf{u}\|,
	\end{equation}
	where the constant $C_D$ involves only the norm of the Sobolev embedding $H^{1,1}(D)\subset L^2(D)$.
	In view of the Poincar\'e inequality, Assumption $\mathbf{H_1},$ (\ref{Equation NS8.1}) and  (\ref{Equation NS8}),  we get that
	\begin{equation}\label{Equation NS9}
	d|\mathbf{v}|^2+(\nu\lambda_N-L)|\mathbf{v}|^2dt\leq \frac{C_{D}^2}{\nu}|\mathbf{v}|^2\|\mathbf{u}\|^2dt+\langle\sigma(\mathbf{u}^x)-\sigma(\tilde{\mathbf{u}}^y),\mathbf{v}\rangle dW_t.
	\end{equation}
	Hence it follows that
	\begin{equation}\label{Equation NS9.1}
	\mathbb{E}[|\mathbf{v}|^2\exp((\nu\lambda_N-L)t-\frac{C_{D}^2}{\nu}\int_{0}^{t}\|\mathbf{u}\|^2ds)]\leq |x-y|^2.
	\end{equation}\par 
	
	Intuitively, if $\exp((\nu\lambda_N-L)t-\frac{C_{D}^2}{\nu}\int_{0}^{t}\|\mathbf{u}\|^2ds)$ increases to infinity, then (\ref{Equation NS9.1}) implies some kind of decay property of $\mathbf{v}$ (with respect to time $t$) such that $\mathbf{u}^x$ and $\tilde{\mathbf{u}}^y$ can be coupled asymptotically. Indeed, although (\ref{Equation NS9.1}) is not able to show  $\mathbf{v}$ decays in $L^2(D)$, it implies that $\mathbf{v}$ goes to zero almost surely.	 Moreover, we can further bound the total variation  distance between $P_t(y,\cdot)$ and $\tilde{\mathbf{u}}^y_t$,  we have the following lemma, which is crucial to prove the eventual continuity.

	\begin{lemma}\label{NS Lemma1}
		Assume $\mathbf{H_1}$ - $\mathbf{H_3}$. Then\\
		$(1)$ $\;\mathbf{v}_t$ converges to $0$ almost surely as $t\rightarrow\infty;$\\
		$(2)$ Let $d_{TV}(\mu,\nu)$ be the total variation distance between $\mu,\nu$, then
		\begin{equation*}\label{eq TVdis} 
		d_{TV}(P_t(y,\cdot),{\rm Law}(\tilde{\mathbf{u}}^y_t))\leq C_1|x-y|^{C_2}\exp(C_3|x|^2),\; t\geq 0.
		\end{equation*}
		where $C_1,C_2,C_3$ are positive constants depending only on $\nu,B_0,L,\lambda_N,C_D,C_0,$ independent of $t.$
		
	\end{lemma}

	Applying Lemma \ref{NS Lemma1}, we are able to deduce the eventual continuity.
	\begin{proposition}\label{NS Prop3}
		Assume $\mathbf{H_1}$ - $\mathbf{H_3}$.
		Then the Markov semigroup $\{P_t\}_{t\geq 0}$ is eventually continuous on $H$. 
	\end{proposition}

	\begin{proof}[Proof.] 
		For any $f\in L_b(H),$
		\begin{equation*}\label{Equation NS14.6}
		\begin{aligned}
		|P_tf(x)-P_tf(y)|&=|\mathbb{E}f(\mathbf{u}^x_t)-\mathbb{E}f(\mathbf{u}^y_t)|\\
		&\leq |\mathbb{E}(f(\mathbf{u}^x_t)-f(\tilde{\mathbf{u}}^y_t))|+|\mathbb{E}(f(\tilde{\mathbf{u}}^y_t)-f(\mathbf{u}^y_t))|\\
		&\leq  \mathbb{E}[(\|f\|_{Lip}|\mathbf{v}_t|)\wedge (2\|f\|_{\infty})]+\|f\|_{\infty}d_{TV}(P_t(y,\cdot),\text{Law}(\tilde{\mathbf{u}}^y_t)).
		\end{aligned}
		\end{equation*}\par 
		
		Hence by Lemma \ref{NS Lemma1} and the Dominated Convergence Theorem,
		\begin{equation*}
		\limsup\limits_{y\rightarrow x}\limsup\limits_{t\rightarrow \infty}|P_tf(x)-P_tf(y)|=0,
		\end{equation*}
		which completes the proof.
		
	\end{proof}

    \begin{remark}
    We mention that due to the dissipation in the equations, the associated Markov semigroup $\{P_t\}_{t\geq 0}$ can be verified to satisfy the e-property or the asymptotic strong Feller (ASF) property. The advantage of the eventual continuity is that it simplifies some finer estimates.

    For instance, Lemma \ref{NS Lemma1} shows that $\mathbf{v}$ decreases to zero almost surely, which is not enough to imply neither the e-property or the ASF. This is because $\mathbb{E}|\mathbf{v}_t|$ may not necessarily decrease to zero, even though $\mathbf{v}$ does almost surely. Actually, if we further require the convergence holds in $L^p(D)$ for some $p\geq 1$, then it follows that
	\begin{equation*}
	    |P_tf(x)-P_tf(y)|\leq \|f\|_{Lip}[\mathbb{E}|\mathbf{v}_t|^p]^{1/p}+C_1\|f\|_{\infty}\exp(C_3{|x|^2})|x-y|^{C_2},
	\end{equation*}
	which implies the ASF immediately. Similarly, the e-property follows, if 
	$\mathbb{E}|\mathbf{v}_t|^p$  is uniformly bounded in $t$. However, the $L^p$-estimate of $\mathbf{v}$ is not readily available.	There are two main difficulties in proving the decay of $\mathbb{E}|\mathbf{v}_t|^p$. On the one hand, as the noise is non-additive, in order to avoid the stochastic integral term for technical reasons, one can only  bound the moments of $\mathbf{v}$ with some multiplier through usual energy estimates.  On the other hand, the extra term in (\ref{Equation NS9.1}), which involves the stronger $\|\cdot\|$-norm of the solution $\mathbf{u}^x$,  hence is not easy to control.

    \end{remark}
    
    It remains to verify the uniform irreducibility to apply Theorem \ref{Thm 4}. This part becomes much difficult when dealing SPDEs driven by degenerate multiplicative noise. Therefore, for simplicity, we consider the noise is additive or nondegenerate for example. Noting that  in \cite{O2008}, it is proved that Equation (\ref{Equation NS}) is exponentially ergodic,  hence the uniform irreducibility for the degenerate multiplicative case is likely to hold true.

	\begin{proposition}
	Assume $\mathbf{H_1}$ - $\mathbf{H_3}$. Further, assume that  the noise is additive, that is,
	\begin{equation*}
	    \sigma(x)=\sigma(0)\;\text{for all}\;x\in H,
	\end{equation*}
	then $\{P_t\}_{t\geq 0}$ is asymptotically stable.
	\end{proposition}
	\begin{proof}[Proof.] 
	As  the noise is additive, the uniform irreducibility follows from \cite[Lemma 3.1]{EM2001}, which completes the proof.
	\end{proof}

	\begin{remark}
	The uniform irreducibility is also easily obtained when the noise is uniformly nondegenerate, which in turn implies the asymptotic stability for Equation (\ref{Equation NS}). For example, if we further make the following hypothesis:
	
	\begin{itemize}
		\item[$\mathbf{H_2'}$] The function $\sigma$ is uniformly nondegenerate, i.e., there exists a constant $C_0$ such that
		\begin{equation*}
		\sup\limits_{\mathbf{u}\in H}|\sigma^{-1}(\mathbf{u})(\mathbf{w})|\leq C_0|\mathbf{w}|.
		\end{equation*}
	\end{itemize}
	\end{remark}

	\paragraph{Uniform irreducibility} In this case, we apply similar arguments as in \cite[Lemma 3.8]{HM2011} to show the uniform irreducibility. 
	
	\begin{proposition}\label{NS Prop UI}
    Assume $\mathbf{H_1}$ and $\mathbf{H_2'}$, then $\{P_t\}_{t\geq 0}$ is uniformly irreducible.
	\end{proposition}
	\begin{proof}[Proof.] 
	The proof resembles that of \cite[Lemma 2.4]{SS2002} and \cite[Lemma 3.8]{HM2011}. Fix $x\in B(0,R)$ and $\epsilon\in(0,1)$, let $\mathbf{w}$ be the solution of deterministic equation
	\begin{equation*}
	\left\{	\begin{aligned}
		&d\mathbf{w}+\mathbf{w}\cdot\nabla\mathbf{w}dt=(\nu\Delta\mathbf{w}-\nabla p)dt,\\
		&\mathbf{w}_0=x+\frac{\epsilon}{2}e_1,
		\end{aligned}\right.
	\end{equation*}
	where $e_1$ is the eigenvector corresponding to eigenvalue $\lambda_1$.
	It is easy to show that 
	\begin{equation*}
		|\mathbf{w}_t|^2\leq e^{-2\nu t}|x+\frac{\epsilon}{2}e|^2 \leq e^{-2\nu t}(R+1)^2.
	\end{equation*}
	Hence, we choose $T=T(R)>0$ sufficiently large such that $|\mathbf{w}_t|<\epsilon/2$. Furthermore, applying some enstrophy estimates, there exists some $K>0$ such that
	\begin{equation*}
	\|\mathbf{w}_t\|\leq K<\infty\text{ for all }t\geq 0.
	\end{equation*}\par 
	Define 
	\begin{equation*}
		D(t):=|\mathbf{u}_t-\mathbf{w}_t|^2-(\frac{\epsilon}{2})^2,
	\end{equation*}
	Then 
	\begin{equation*}
		dD(t)=-2\nu\langle\mathbf{u}_t-\mathbf{w}_t,\Delta(\mathbf{u}_t-\mathbf{w}_t)+B(\mathbf{u}_t,\mathbf{u}_t)-B(\mathbf{w}_t,\mathbf{w}_t)\rangle dt+|\sigma(\mathbf{u}_t)|^2dt+2\langle \mathbf{u}_t-\mathbf{w}_t, \sigma(\mathbf{u}_t)dW_t\rangle,
	\end{equation*}
	where $B(u,v):=-u\cdot\nabla v$ and  $D(0)=0$. Define 
	\begin{equation*}
		\tau:=\inf\{t\geq 0:|D(t)|>(\frac{\epsilon}{4})^2\}.
	\end{equation*}
	For $t\in [0,\tau)$,
	\begin{equation*}
		\frac{3\epsilon^2}{16}<|\mathbf{u}_t-\mathbf{w}_t|^2< \frac{5\epsilon^2}{16},\quad |\mathbf{u}_t|^2<2(|\mathbf{u}_t-\mathbf{w}_t|^2+|\mathbf{w}_t|^2)\leq 2(R+1)^2+ \frac{5\epsilon^2}{8},
	\end{equation*}
	and that
	\begin{equation*}
		-2\nu\langle\mathbf{u}_t-\mathbf{w}_t,\Delta(\mathbf{u}_t-\mathbf{w}_t)\rangle=-2\nu\|\mathbf{u}_t-\mathbf{w}_t\|^2,
	\end{equation*}
	\begin{equation*}
	\begin{aligned}
	\langle\mathbf{u}_t-\mathbf{w}_t,B(\mathbf{u}_t,\mathbf{u}_t)-B(\mathbf{w}_t,\mathbf{w}_t)\rangle&=\langle\mathbf{u}_t-\mathbf{w}_t,B(\mathbf{u}_t,\mathbf{u}_t-\mathbf{w}_t)\rangle+\langle\mathbf{u}_t-\mathbf{w}_t,B(\mathbf{u}_t-\mathbf{w}_t,\mathbf{w}_t)\rangle\\
	&\leq C_D|\mathbf{u}_t-\mathbf{w}_t|\,\|\mathbf{u}_t-\mathbf{w}_t\|\,\|\mathbf{w}_t\|\\
	&\leq \nu \|\mathbf{u}_t-\mathbf{w}_t\|^2+\frac{C_D^2}{4\nu}|\mathbf{u}_t-\mathbf{w}_t|^2\,\|\mathbf{w}_t\|^2.
	\end{aligned}
	\end{equation*} 
	In summary, we obtain that for $t\in[0,\tau)$,
	\begin{equation*}
	    -2\nu\langle\mathbf{u}_t-\mathbf{w}_t,\Delta(\mathbf{u}_t-\mathbf{w}_t)+B(\mathbf{u}_t,\mathbf{u}_t)-B(\mathbf{w}_t,\mathbf{w}_t)\rangle dt+|\sigma(\mathbf{u}_t)|^2dt\leq (\frac{5}{64}C_D^2 K^2 \epsilon^2 +B_0)dt.
	\end{equation*}
	For the stochastic part, letting $M_t:=\int_{0}^{t}2\langle \mathbf{u}_s-\mathbf{w}_s, \sigma(\mathbf{u}_s)dW_s\rangle $  be the local martingale part, we infer that
	\begin{equation*}
		\frac{d}{dt}\langle M\rangle_t=4|\langle \mathbf{u}_t-\mathbf{w}_t,\sigma(\mathbf{u}_t)|^2\leq 4 B_0|\mathbf{u}_t-\mathbf{w}_t|^2,
	\end{equation*}
	 On ther other hand, by $\mathbf{H_2'}$, the inverse operator $\sigma^{-1}(\mathbf{u}_t)$ is well defined and bounded. Consequently, there exists a constant $C_0>0$ such that
	\begin{equation*}
	\begin{aligned}
		\frac{d}{dt}\langle M\rangle_t&=4|\langle \mathbf{u}_t-\mathbf{w}_t,\sigma(\mathbf{u}_t)\rangle|^2\geq 4 C_0|\mathbf{u}_t-\mathbf{w}_t|^2>\frac{3}{4}C_0\epsilon^2>0.
	\end{aligned}
	\end{equation*}\par 	
	Let now $\widetilde{W}$ be a Wiener	process that is independent of $W$ and define
	\begin{equation*}
		Y(t):=D(t\wedge \tau)+(\widetilde{W}_t-\widetilde{W}_\tau)\mathbf{1}_{\{t\geq\tau\}}.
	\end{equation*}
	Then $Y(t)$ is a semimartingale with $Y(0)=0$, which fulfills the conditions of \cite[Lemma I.8.3]{B1998} . Therefore, there exists $p>0$ such that for all $x\in B(0,R)$,
	\begin{equation*}
		\mathbb{P}(\sup\limits_{0\leq t\leq T}|Y(t)|\leq (\frac{\epsilon}{4})^2)\geq p.
	\end{equation*}\par 
	Then
	\begin{equation*}
		P_T(x,B(0,\epsilon))=\mathbb{P}(|\mathbf{u}_T|<\epsilon)\geq \mathbb{P}(\tau> T)\geq 	\mathbb{P}(\sup\limits_{0\leq t\leq T}|Y(t)|\leq (\frac{\epsilon}{4})^2)\geq p
	\end{equation*}
	for all $x\in B(0,R),$ completing the proof.
	\end{proof}

	In conclusion, we have the following theorem.
	\begin{theorem}
	    Assume $\mathbf{H_1}$ and $\mathbf{H_2'}$, then $\{P_t\}_{t\geq 0}$ is asymptotically stable.
	\end{theorem}

    \begin{remark}
    We point that  the eventual continuity method is also applicable to the stochastic delay equations, studied in \cite{HM2011}. As in  \cite{HM2011}, we can obtain the uniqueness of invariant measures can be obtained through he one-sided Lipschitz and non-degeneracy condition.  While these conditions are not able to ensure either the tightness or the lower bound condition, therefore, the existence of invariant measures require some extra conditions.
    \end{remark}

    \subsection{The modified  Lagrangian observation process} \label{Section L}
	
	Next example is the modified  Lagrangian observation process with multiplicative noise. For the Lagrangian observation process with additive noise, in \cite{KPS2010}, Komorowski, Peszat and Szarek use  Malliavin calculus to prove the e-property, and further show the unique ergodicity. For our case,  as the driven noise is non-additive, the Malliavin calculus approach is inconvenient to apply. Instead, we use the asymptotic coupling arguments used in \cite{HMS2011,NJG2017,KS2018} to establish the eventual continuity.\par 

	 We first introduce some standard notations.	Given an $r\geq 0$, we denote by $\mathcal{X}^r$ the Sobolev space which is the completion of 
	\begin{equation*}
		\{x\in C^\infty(\mathbb{T}^d;\mathbb{R}^d):\int_{\mathbb{T}^d}x(\xi)d\xi=0,\hat{x}(k)\in\text{ Im } \mathcal{E}(k),\,\forall k\in \mathbb{Z}_*^d\}
	\end{equation*}
	with respect to the norm 
	\begin{equation*}
		\|x\|^2_{\mathcal{X}^r}:=\sum_{k\in\mathbb{Z}_*^d}|k|^{2r}|\hat{x}(k)|^2,
	\end{equation*}
	where $\mathbb{Z}_*^d=\mathbb{Z}^d-\{0\}$ and
	\begin{equation*}
		\hat{x}(k):=(2\pi)^{-d}\int_{\mathbb{T}^d}x(\xi)e^{-i\xi\cdot k}d\xi,\;k\in\mathbb{Z}_*^d,
	\end{equation*}
	are the Fourier coefﬁcients of $x$. \par 
	Let $A_r$ be an operator on $\mathcal{X}^r$ defined by
	\begin{equation*}
		\widehat{A_rx}(k):=-\gamma(k)\hat{x}(k),\;k\in\mathbb{Z}_*^d,
	\end{equation*}
	with the domain
	\begin{equation*}
		D(A^r):=\{x\in\mathcal{X}^r:\sum_{k\in\mathbb{Z}_*^d}|\gamma(k)|^2|k|^{2r}|\hat{x}(k)|^2<\infty\}.
	\end{equation*}\par 
	For each $u\in \mathcal{X}^r$, let $Q(u)$ be a symmetric positive deﬁnite bounded linear operator on
	\begin{equation*}
		\{x\in L^2(\mathbb{T}^d,d\xi;\mathbb{R}^d):\int_{\mathbb{T}^d}x(\xi)d\xi=0 \}
	\end{equation*}
	given by
	\begin{equation*}
		\widehat{Q(u)x}(k):=q_k(u)\gamma(k)\mathcal{E}(k)\hat{x}(k),
	\end{equation*}
	where $q_k:\mathcal{X}^r\rightarrow (0,\infty)$ which we will specify below.

	Let $m\geq 0$ be a constant and let $\mathcal{X}:=\mathcal{X}^m$  and  $\mathcal{V}:=\mathcal{X}^{m+1}$. Note that, by Sobolev embedding, $\mathcal{X}\hookrightarrow C^1(\mathbb{T}^d,\mathbb{R}^d)$ and hence there exists a constant $C>0$ such that
	\begin{equation}\label{Equation L1}
		\|x\|_{C^1(\mathbb{T}^d,\mathbb{R}^d)}\leq C\|x\|_{\mathcal{X}},\;\forall x\in\mathcal{X}.
	\end{equation}\par 
	For a given $x\in \mathcal{X}$ and a cylindrical Wiener process $W$ in $\mathcal{X}$ , consider the SPDE
	\begin{equation}\label{Equation L0}
	\left\{\begin{aligned}
	&d\mathbf{u}_t=A\mathbf{u}_tdt +B(\mathbf{u}_t,\mathbf{u}_t)dt+Q^{1/2}(\mathbf{u}_t)dW_t,\\
	&\mathbf{u}_0=x,
	\end{aligned}\right.
	\end{equation}
	where $W$ is a cylindrical Wiener process in $\mathcal{X}$ and 
	\begin{equation*}
		B(\psi,\phi)(\xi):=\psi(0)\cdot\nabla\phi=(\sum_{j=1}^{d}\psi_j(0)\frac{\partial\phi_1}{\partial\xi_j}(\xi),\dots,\sum_{j=1}^{d}\psi_j(0)\frac{\partial\phi_d}{\partial\xi_j}(\xi)),\;\psi,\phi\in\mathcal{X},\xi\in\mathbb{T}^d.
	\end{equation*}
	By (\ref{Equation L1}), $B(\cdot,\cdot)$ is a continuous bilinear form mapping from $\mathcal{X\times\mathcal{X}}$ into $\mathcal{X}^{m-1}$.\par 
	Given $x\in\mathcal{X}$, let $\mathbf{u}^x_t$ denote the value at $t\geq 0$ of a solution to (\ref{Equation L0}) satisfying $\mathbf{u}_0^x=x$. Since the existence of a strong solution follows from the Banach fixed point argument, $\mathbf{u}=\{\mathbf{u}^x,x\in\mathcal{X}\}$ is a stochastically continuous Markov family and its transition semigroup $\{P_t\}_{t\geq 0}$ is Feller.\par 
	
	We make the following Hypotheses: 
	\begin{itemize}
		\item[$\mathbf{L_1.}$] 	
		\begin{equation*}
			\exists m>d/2+1,\alpha\in(0,1)\quad |||\mathcal{E}|||:=\sum_{k\in\mathbb{Z}_*^d}\gamma^\alpha(k)|k|^{2(m+1)}\text{Tr }\mathcal{E}(k)<\infty
		\end{equation*}
		\item[$\mathbf{L_2.}$] 
		\begin{equation*}
			\int_{0}^{\infty}\sup\limits_{k\in\mathbb{Z}^d_*}e^{-\gamma(k)t}|k|dt<\infty.
		\end{equation*}		\par 
		\item[$\mathbf{L_3.}$] There exists $M\in\mathbb{N}$ such that for $|k|>M,$
		\begin{equation*}
		q_k(x)\equiv 1\quad\forall x\in\mathcal{X}.
		\end{equation*}
		For $|k|\leq M,$ there exists constants $0<\alpha<\beta<\infty$ such that
		\begin{equation*}
			\alpha\leq q_k(x)\leq \beta\quad\forall x\in\mathcal{X}.
		\end{equation*}\par 
		\item[$\mathbf{L_4.}$] Furthermore, we assume that $q_k(\cdot)$ is defined only by the low modes, that is, 
		\begin{equation*}
			q_{k}(x)=q_k(P_Mx)\quad\forall x\in\mathcal{X},\,\forall k\in\mathbb{Z}_*^d.
		\end{equation*}
		We also assume that $q_k^{1/2}$ is Lipschitz, i.e. there exists some constant $K>0$ such that
		\begin{equation*}
		|q_k^{1/2}(x)-q_k^{1/2}(y)|\leq K|x-y| \quad\forall x,y\in\mathcal{X},\,\forall k\in\mathbb{Z}_*^d.
		\end{equation*}
		\end{itemize}

	\begin{theorem}\label{Thm L} Assume $\mathbf{L_1}$-$\mathbf{L_4}$. Then the Markov semigroup	$\{P_t\}_{t\geq 0}$, generated by $\mathbf{u}=\{\mathbf{u}^x,x\in\mathcal{X}\}$, is asymptotically stable.
	\end{theorem}\par

	\begin{remark}
	If we further assume that $q_k\equiv1$ for all $x\in\mathbb{Z}_*^d$, then the driven noise turns into the additive case as in \cite{KPS2010}. Therefore, our result can be seen as a generalization of \cite[Theorem 5]{KPS2010}.
	\end{remark}

	To prove Theorem \ref{Thm L},  we verify the eventual continuity,  the energy estimates, and the uniform irreducibility.
	
    \paragraph{Eventual continuity}
	Again, we apply the asymptotic coupling argument to prove the eventual continuity, and overcome the difficulties caused by the multiplicative noise.\par 
	Fix any $x,y\in \mathcal{X}$, take $\mathbf{u}^x$ solving (\ref{Equation L0}) with initial datum $x,$ and define $\tilde{\mathbf{u}}^y$ as the solution to the following  equation (\ref{Equation L0}) with the initial condition $y$ and the additional
	control term: 
	\begin{equation}\label{Equation L6}
	\left\{\begin{aligned} 
	&d\tilde{\mathbf{u}}_t^y=A\tilde{\mathbf{u}}_t^ydt +B(\tilde{\mathbf{u}}_t^y,\tilde{\mathbf{u}}_t^y)dt+Q^{1/2}(\tilde{\mathbf{u}}^y_t)dW_t+P_N(\lambda(\mathbf{u}_t^x-\tilde{\mathbf{u}}_t^y)+B(\mathbf{u}_t^x,\mathbf{u}_t^x)-B(\tilde{\mathbf{u}}_t^y,\tilde{\mathbf{u}}_t^y))dt\\
	&\tilde{\mathbf{u}}_0=y,
	\end{aligned}\right.
	\end{equation}
	where $N>M$ which we will specify below, and $P_N$ is the projection operator. \par 
	Let $\mathbf{v}:=\mathbf{u}^x-\tilde{\mathbf{u}}^y.$ We have that
	\begin{equation}\label{Equation L7}
	\left\{\begin{aligned}
	&d\mathbf{v}_t=A\mathbf{v}_tdt-\lambda P_N\mathbf{v}_tdt-Q_N(B(\mathbf{u}_t^x,\mathbf{u}_t^x)-B(\tilde{\mathbf{u}}_t^y,\tilde{\mathbf{u}}_t^y))dt+(Q^{1/2}(\mathbf{u}^x_t)-Q^{1/2}(\tilde{\mathbf{u}}^y_t))dW_t,\\
	&\mathbf{v}_0=x-y.
	\end{aligned}\right.
	\end{equation}\par 
	We proceed to establish some estimates on $\mathbf{v}$ needed for existence and uniqueness of invariant measures. Applying the It\^o lemma to (\ref{Equation L7}), we find that
	\begin{equation*}
	\begin{aligned}
	d|\mathbf{v}_t|^2&=2\langle\mathbf{v}_t,\mathbf{v}_t\rangle dt-2\lambda|P_N\mathbf{v}_t|^2dt+2\langle B(\mathbf{v}_t,Q_N\mathbf{u}^x_t),\mathbf{v}_t\rangle dt+|Q^{1/2}(\mathbf{u}^x_t)-Q^{1/2}(\tilde{\mathbf{u}}^y_t)|^2dt\\
	&\quad+2\langle\mathbf{v}_t,(Q^{1/2}(\mathbf{u}^x_t)-Q^{1/2}(\tilde{\mathbf{u}}^y_t))dW_t\rangle\\
	&\leq -2\gamma_*|\mathbf{v}_t|^2dt+(K^2\|Q\|-\lambda)|P_N\mathbf{v}_t|^2dt+C\|Q_N\mathbf{u}\|\,|\mathbf{v}_t|^2dt+g_t,
	\end{aligned}
	\end{equation*}
	where $g_t:=2\langle\mathbf{v}_t,(Q^{1/2}(\mathbf{u}^x_t)-Q^{1/2}(\tilde{\mathbf{u}}^y_t))dW_t\rangle,$ and we use a generic bound 
	\begin{equation*}
	|\langle B(\mathbf{x}_t,\mathbf{y}_t),\mathbf{z}_t\rangle |\leq|\mathbf{x}_t(0)|\, \|\mathbf{y}_t\|\,|\mathbf{z}_t|\leq C|\mathbf{x}_t|\, \|\mathbf{y}_t\|\,|\mathbf{z}_t|.
	\end{equation*}
	Define
	\begin{equation*}
	h(z):=\frac{z^2}{1+\gamma_*^{-1}|z|^2},\quad z\geq 0.
	\end{equation*}
	Note that there exists a constant $\widetilde{C}$ such that
	\begin{equation*}
	Czw^2\leq \gamma_*w^2+\widetilde{C}h(z)w^2,\quad z\geq 0,w\in\mathbb{R}.
	\end{equation*}
	Therefore,
	\begin{equation}\label{Equation L8}
	d|\mathbf{v}_t|^2\leq -\gamma_*|\mathbf{v}_t|^2dt+(K^2\|Q\|-\lambda)|P_N\mathbf{v}_t|^2dt+\widetilde{C} h(\|Q_N\mathbf{u}_t^x\|)|\mathbf{v}_t|^2dt+g_t.
	\end{equation}\par 
	Furthermore, we obtain that 
		\begin{lemma}\label{Lemma L3}
		Assume $\mathbf{L_1}$-$\mathbf{L_4}$. Then
		\begin{itemize}
			\item[$(1)$] 	$\mathbf{v}_t$ converges to $0$ almost surely as $t\rightarrow\infty$;
			\item[$(2)$] \begin{equation}\label{Equation L11}
			d_{TV}(P_t(y,\cdot),{\rm Law}(\tilde{\mathbf{u}}^y_t))\leq C_2|x-y|^{2},\quad t\geq 0.
			\end{equation}
		\end{itemize}
		where $c,C_1,C_2$ are positive constants depending only on $\mathcal{E},\gamma_*,\alpha,\beta,K,M$, independent of $t$.
		\end{lemma}
	Applying Lemma \ref{Lemma L3}, we are able to deduce the eventual continuity as in Proposition \ref{NS Prop3}.
	\begin{proposition}\label{Prpo L1}
		Assume  $\mathbf{L_1}$-$\mathbf{L_4}$.
		Then the Markov semigroup $\{P_t\}_{t\geq 0}$ is eventually continuous on $\mathcal{X}$. 
	\end{proposition}

	\paragraph{Uniform irreducibility and Energy estimates}
	
	Finally, we have the following lemmas, and these complete the  proof of Theorem \ref{Thm L}.
	\begin{lemma}\label{Lemma L1}
		\begin{equation}
			\mathbb{E}|\mathbf{u}_t|^2\leq e^{-2\gamma_*t}|x|^2+C|||\mathcal{E}|||
		\end{equation}
		for some positive constant $C$. 
	\end{lemma}
	\begin{lemma}\label{Lemma L2}
	For any $\epsilon>0,\,R>0$, there exists $T=T(\epsilon,R)>0$ such that
	\begin{equation*}
	\inf\limits_{x\in B(0,R)}P_T(x,B(0,\epsilon))>0.
	\end{equation*} 
	\end{lemma}

	\subsection{The 2D hydrostatic	Navier--Stokes equation}\label{Section HNS}

	Next we consider the 2D stochastic  hydrostatic	Navier--Stokes equations with multiplicative noise. Fix $L,h > 0$ and consider domain $D$ defined by $D:=\{(z_1,z_2)\in\mathbb{R}^2:z_1\in(0,L),z_2\in(-h,0)\}$. The boundary $\partial D$ is decomposed into its lateral side $\Gamma_l:=\{0,L\}\times[-h,0]$ and horizontal side $\Gamma_h:=[0,L]\times\{-h,0\}$.

	\begin{equation}\label{Equation HNS}
	\left\{\begin{aligned}
	&du+(u\partial_{z_1}u+w\partial_{z_2}u-\nu\Delta u+\partial_{z_1}p)dt=\sum_{k=1}^{m}\sigma_k(u_t)dW^k_t,\\
	&\partial_{z_2}p=0,\quad \partial_{z_1}u+\partial_{z_2}w=0,\quad u(0)=x,\\
	&u|_{\Gamma_l}=0,\quad \partial_{z_2}u|_{\Gamma_h}=w|_{\Gamma_h}=0.
	\end{aligned}\right.
	\end{equation}
	where  $m\in\mathbb{N}$, $W=(W^1,\dots,W^m)$ is a standard $m$-dimensional Brownian motion, $\nu>0$.\par 
	Introduce the following spaces
	\begin{equation*}
	\begin{aligned}
	H:&=\{\psi\in L^2(D):\int_{-h}^{0}\psi(z_1,z_2)dz_2=0\text{ for all }z_1\in(0,L)\},\\
	V:&=\{\psi\in H^1(D):\int_{-h}^{0}\psi(z_1,z_2)dz_2=0\text{ for all }z_1\in(0,L) \text{ and } \psi|_{\Gamma_l}=0\},
	\end{aligned}
	\end{equation*}
	and denote the norms associated to $H$ and $V$ respectively as $|\cdot|$ and $\|\cdot\|$. 	The eigenvectors $e_1,e_2,\dots,$ of the negative Laplacian operator with boundary conditions given by (\ref{Equation HNS}) (with the corresponding eigenvalues $0<\lambda_1<\lambda_2<\dots$) form a complete orthonormal basis of $V$.\par 
	We make the following Hypotheses: 
	\begin{itemize}
		\item[$\mathbf{HNS_1.}$] 	 For $u\in H$, $k=1,\dots,m,$ $\sigma_k(u)\in H^2(D)$, $\sigma_k(u)|_{\Gamma_l}=0$, $\partial_{z_2}\sigma_k(u)|_{\Gamma_h}=0$ and $\int_{-h}^{0}\sigma_k(u)dz_2\equiv 0$. Furthermore, assume that there exist constants $B_0,L>0$ such that 
	\begin{equation*}
		\begin{aligned}
		&\|\sigma(u)\|^2+\|\partial_{z_2}\sigma(u)\|^2=\sum_{k=1}^{m}(\|\sigma_k(u)\|^2|+\|\partial_{z_2}\sigma_k(u)\|^2)\leq B_0,\text{ for all }u\in V;\\
		&|\sigma(u)-\sigma(v)|^2=\sum_{k=1}^{m}|\sigma_k(u)-\sigma_k(v)|^2\leq L|u-v|^2,\text{ for all }u,v\in V.
		\end{aligned}
		\end{equation*}
		\item[$\mathbf{HNS_2.}$] There exists $N\in\mathbb{N}$ such that
		\begin{equation*}
		P_NH\subset \text{Range}\;(\sigma_k(u)),\text{ for all }u\in H,\;k=1,\dots,m.
		\end{equation*}\par 
		Moreover, the corresponding pseudo-inverse operators $\sigma_k^{-1}:P_NV\rightarrow V$ are uniformly bounded, i.e., there exists a constant $C_0$ such that
		\begin{equation*}
		\|\sigma(u)^{-1}(P_Nw)\|\leq C_0\|P_Nw\|,\text{ for all }u,w\in V.
		\end{equation*}\par 
		\item[$\mathbf{HNS_3.}$] 	Finally, we assume that $N$ is sufficiently large such that
		\begin{equation*}\label{Equation HNS0}
		\lambda_N >C(B_0+L+1),
		\end{equation*} 
		for some constant $C>0$.
	\end{itemize}

	The existence of a unique strong solution to equation (\ref{Equation HNS}) can be argued by the same way as in \cite[Section 3.2]{NJG2017} under the above assumptions on $\sigma$.  Moreover, $u$ is a  Markov Feller process with the state space $V$. Different from (\ref{Equation NS}), the solution of (\ref{Equation HNS}) lives in a more regular space $V$, and  it causes some technical problems. First, as the energy estimates involve the stronger norm, the $V$-norm,  is to obtain an (\ref{eq lyapunov})-like estimate. However, the weaker form of the energy estimates  is straightforward to derive, which motivates us to apply Theorem \ref{Thm 5} to prove the unique ergodicity.
	Second, it seems  difficult to verify the eventual continuity directly through similar construction and estimates as in Subsection \ref{Section NS}. Nevertheless, instead of the eventual continuity, we can adopt the general form, the $d$-eventual continuity, to avoid these difficulties.	We endow $V$ with a new norm $d$, the $H$-norm. Then we can derive the unique ergodicity via Theorem \ref{Thm 5}.   \par   
	
	\paragraph{Energy estimates} Applying the Ito lemma, applying the same argument as in Proposition \ref{NS Prop EE}, it follows that
	\begin{lemma}\label{HNS Lemma EE} Assume $\mathbf{HNS_1}$. Then
		\begin{equation}\label{Equation HNS20}
	    \mathbb{E}|u|^2\leq e^{-2\nu t}|x|^2+\frac{B_0}{2\nu}.
		\end{equation}
	\end{lemma}

	Next we prove the $d$-eventual continuity.
	\paragraph{Eventual continuity}
	Now we construct the generalized coupling for (\ref{Equation HNS}), which is the same as in  \cite[Section 6.2.2]{KS2018} and \cite[Section 4.4]{KS2020}, and is slightly modified from the construction in \cite[Section 3.2.4]{NJG2017}.	Fix $x,y\in V$,  let $\tilde{u}$ be the	solution to a similar system with the first equation changed to
	
	\begin{equation*}
	d\tilde{u}+(\tilde{u}\partial_{z_1}\tilde{u}+\tilde{w}\partial_{z_2}\tilde{u}-\nu\Delta \tilde{u}+\partial_{z_1}\tilde{p}+\frac{\nu\lambda_N}{2}P_N(u-\tilde{u}))dt=\sum_{k=1}^{m}\sigma_k(\tilde{u})dW^k.
	\end{equation*}

	For the difference $v:=u-\tilde{u}$, by similar estimates as in \cite[Section 3.2.4]{NJG2017} and \cite[Section 4.4]{KS2020},  one has
	\begin{equation}\label{Equation HNS3}
		\mathbb{E}[|v(t)|^2\exp((\nu\lambda_N-L) t-C_1\int_{0}^{t}(\|u\|^2+\|\partial_{z_2}u\|^2)ds)]\leq |x-y|^2,
	\end{equation}
	with a constant $C_1$ depending only on $\nu$ and $D$. Indeed, we have the following lemma.
		
	\begin{lemma}\label{HNS Lemma1}
		Assume $\mathbf{HNS_1}$-$\mathbf{HNS_3}$.
		Then\\
		$(1)$ $\;v(t)$ converges to $0$ almost surely as $t\rightarrow\infty;$\\
		$(2)$ 
		\begin{equation*}
		d_{TV}(P_t(y,\cdot),{\rm Law}(\tilde{u}(t)))\leq c_1|x-y|^{c_2}\exp(c_3(|x|^2+|\partial_{z_2}x|^2)),\; t\geq 0.
		\end{equation*}
		where $c_1,c_2,c_3$ are positive constants, independent of $t$.
		
	\end{lemma}
	
	\begin{proposition}\label{HNS Prop2}
		Assume $\mathbf{HNS_1}$ - $\mathbf{HNS_3}$.
		Then the Markov semigroup $\{P_t\}_{t\geq 0}$ is eventually continuous on $V$. 
	\end{proposition}

	Finally, we show the uniform irreducibility. Again, we prove it under the uniform nondegenerate assumption.
	
	\begin{itemize}
		\item[$\mathbf{HNS_2'}$] The function $\sigma$ is uniformly nondegenerate, i.e., there exists a constant $C_0$ such that for all $w\in V$
		\begin{equation*}
		\sup\limits_{u\in V}\|\sigma^{-1}(u)(w)\|\leq C_0\|w\|.
		\end{equation*}
	\end{itemize}

	\paragraph{Uniform irreducibility} Still, we apply similar arguments as in \cite[Lemma 3.8]{HM2011} to show the uniform irreducibility. 
	\begin{proposition}\label{HNS Prop UI}
	Assume that $\mathbf{HNS_1}$ and $\mathbf{HNS_2'}$, then $\{P_t\}_{t\geq 0}$ is unformly irreducibile.
	\end{proposition}
	
	Collecting these results, we can prove the unique ergodicity by Theorem \ref{Thm 5}.

	\begin{theorem} \label{HNS Thm}
		Assume $\mathbf{HNS_1}$ and $\mathbf{HNS_2'}$.
		Then the Markov semigroup $\{P_t\}_{t\geq 0}$ is $d$-eventually continuous on $V$. Moreover, $\{P_t\}_{t\geq 0}$  has unique invariant meausre $\mu$ on $V$, such that $P_t(x,\cdot)$ converges weakly to $\mu$ in the $L^2$-topology as $t\rightarrow\infty$.
	\end{theorem}
	
	\begin{proof}[Proof.] 

	We first verify the   the $d$-eventual continuity. Fix any $\varphi\in C_b(V)\cap L_b(V,d),$
		\begin{equation*}
		\begin{aligned}
		|P_t\varphi(x)-P_t\varphi(y)|&=|\mathbb{E}\varphi(u(t))-\mathbb{E}\varphi(\tilde{u}(t))|\\
		&\leq\mathbb{E}[(\|\varphi\|_{Lip}|v(t)|)\wedge (2\|\varphi\|_{\infty})]+\|\varphi\|_{\infty}d_{TV}(P_t(y,\cdot),\text{Law}(\tilde{u}(t))).
		\end{aligned}
		\end{equation*}
		Hence by Lemma \ref{HNS Lemma1} and the Dominated Convergence Theorem, it follows that
		\begin{equation*}
		\limsup\limits_{y\rightarrow x}\limsup\limits_{t\rightarrow \infty}|P_t\varphi(x)-P_t\varphi(y)|=0,
		\end{equation*}
		which implies the $d$-eventual continuity.\par 
		
		Next, by Proposition \ref{prop lbc}, the lower bound condition (\ref{eq 4.1}) follows from Lemma \ref{HNS Lemma EE} and Proposition \ref{HNS Prop UI},  completing the proof by Theorem $\ref{Thm 5}$.
		
	\end{proof}

	\subsection{The damped stochastically forced Euler--Voigt model}\label{Section EV}

	Our next example is the	2D damped stochastically forced Euler--Voigt model. Fix $\gamma>2/3$ and consider the following	equation on the periodic box  $D=[-\pi,\pi]^2$.	
	\begin{equation}\label{Equation EV}
	\left\{\begin{aligned}
	&d\mathbf{u}+(\nu\mathbf{u}+\mathbf{u}_\gamma\cdot\nabla\mathbf{u}_\gamma+\nabla p )dt=\sum_{k=1}^{m}\sigma_k(\mathbf{u}_t)dW^k_t,\\
	&\mathbf{u}(0)=x,\quad\nabla\cdot\mathbf{u}=0,\quad\Lambda^\gamma\mathbf{u}_\gamma=\mathbf{u},
	\end{aligned}\right.
	\end{equation}
	where 	 $\mathbf{u}=(u_1,u_2)$ is the unknown velocity field, $p$ is the unknown pressure, $\nu>0,$  $m\in\mathbb{N}$, $W=(W^1,\dots,W^m)$ is a standard $m$-dimensional Brownian motion.  Here $\Lambda_\gamma=(-\Delta)^{\gamma/2}$ is
	the fractional Laplacian with $\gamma>2/3$, and we use the notation $\|\psi\|_{H^s}:=\|\Lambda^s\|_{L^2},\,s\in\mathbb{R}$. \par 
	
	Introduce the following space:
	\begin{equation*}
	\begin{aligned}
	V&:=\{\psi\in H^{1-\gamma/2}(D)^2:\nabla\cdot\psi=0,\;\int_D\psi = 0\} .
	\end{aligned}
	\end{equation*}
	The eigenvectors $e_1,e_2,\dots,$ of the operator $\Delta|_V$ (with the corresponding eigenvalues $0<\lambda_1<\lambda_2<\dots$) form a complete orthonormal basis of $V$. \par

    We make the following Hypotheses: 
	\begin{itemize}
		\item[$\mathbf{EV_1.}$] 	Assume that for $\mathbf{u}\in V$, $k=1,\dots,m$, $\sigma_k(\mathbf{u})\in H^{1-\gamma/2}(D)^2$, $\nabla\cdot\sigma_k(\mathbf{u})=0$, $\int_D\sigma_k\mathbf{u}=0.$ For $\rho(\mathbf{u}):=(\nabla^{\bot}\cdot\mathbf{\sigma})(\mathbf{u}),$
		assume that there exist constants $B_0,L>$ such that
		\begin{equation*}
		\begin{aligned}
		&\|\rho(\mathbf{u})\|^2_{H^{1-\gamma/2}}\leq B_0,\text{ for all }\mathbf{u}\in H;\\
		&|\sigma(\mathbf{u})-\sigma(\mathbf{v})|^2=\sum_{k=1}^{m}|\sigma_k(\mathbf{u})-\sigma_k(\mathbf{v})|^2\leq L|\mathbf{u}-\mathbf{v}|^2,\text{ for all }\mathbf{u},\mathbf{v}\in H.
		\end{aligned}
		\end{equation*}
		\item[$\mathbf{EV_2.}$] There exists $N\in\mathbb{N}$ such that
		\begin{equation*}
		P_NH\subset \text{Range}\;(\sigma_k(\mathbf{u})),\text{ for all }\mathbf{u}\in H,\;k=1,\dots,m.
		\end{equation*}\par 
		Moreover, the corresponding pseudo-inverse operators $\sigma_k^{-1}:P_NH\rightarrow H$ are uniformly bounded, i.e., there exists a constant $C_0$ such that
		\begin{equation*}
		|\sigma(\mathbf{u})^{-1}(P_N\mathbf{w})|\leq C_0|P_N\mathbf{w}|,\text{ for all }\mathbf{u},\mathbf{w}\in H.
		\end{equation*}\par 
		\item[$\mathbf{EV_3.}$] 	Finally, we assume that $N$ is sufficiently large such that
		\begin{equation}\label{Equation EV5}
    	\lambda_N^{\gamma/2-1/3} >C\nu ^{-3}\|\rho \|^2_{H^{-\gamma/2}},
		\end{equation}
		for some constant $C>0$.
	\end{itemize}

    The existence of a unique strong solution to equation (\ref{Equation EV}) can be argued by the same way as in \cite[Section 3.4]{NJG2017} under the above assumptions on $\sigma$. Moreover, $\mathbf{u}$ is a  Markov process with the state space $V$.
	However, the Feller property of $\mathbf{u}$ is obtained only with respect to a weaker $H^{-\gamma/2}$-norm rather than $H^{1-\gamma/2}$-norm. Likewise,  we  endow $V$ with $H^{1-\gamma/2}$-norm, and denote
	\begin{equation*}
		d(x,y):=\|x-y\|_{H^{-\gamma/2}},\quad\text{ for }x,y\in V,
	\end{equation*}
	then the Markov semigroup $\{P_t\}_{t\geq 0}$ generated by $\mathbf{u}$ is $d$-Feller. Therefore, we can derive the asymptotic stability via Theorem \ref{Thm 5}.

	\paragraph{Energy estimates} Applying the Ito lemma, applying the same argument as in Proposition\ref{NS Prop EE}, it follows that
	\begin{lemma}\label{EV Lemma EE} Assume $\mathbf{HNS_1}$. Then
		\begin{equation}\label{Equation EV20}
	    \mathbb{E}\|\mathbf{u}\|_{H^{-\gamma/2}}^2\leq e^{-2\nu t}|x|^2+\frac{B_0}{2\nu}.
		\end{equation}
	\end{lemma}

	Next we prove the $d$-eventual continuity.\par 
	
	\paragraph{$d$-Eventual continuity}

	We again use the same coupling construction as in \cite[Section 6.2.4]{KS2018} and \cite[Section 4.6]{KS2020}, which is a slightly modified from the construction in \cite[Section 3.4.4]{NJG2017}. Fix $x,y\in V,$ and and $\tilde{\mathbf{u}}$ be the solution to the following equation
	\begin{equation}\label{Equation EV0}
	\left\{\begin{aligned}
	&d\tilde{\mathbf{u}}+(\nu\tilde{\mathbf{u}}+\tilde{\mathbf{u}}_\gamma\cdot\nabla\tilde{\mathbf{u}}_\gamma+\nabla \tilde{p}-\frac{1}{2}\nu\lambda_N^{\gamma/2-1/3} P_N(\mathbf{u}-\tilde{\mathbf{u}}))dt=\sum_{k=1}^{m}\sigma_k(\tilde{\mathbf{u}})dW^k,\\
	&\tilde{\mathbf{u}}(0)=y,\quad\nabla\cdot\tilde{\mathbf{u}}=0,\quad\Lambda^\gamma\tilde{\mathbf{u}}_\gamma=\tilde{\mathbf{u}},
	\end{aligned}\right.
	\end{equation}
	where $P_N$ denotes the projector corresponding to the eigenfunctions of $\Delta|_V$.
	
	Let $\mathbf{v}:=\mathbf{u}-\tilde{\mathbf{u}}$, then by similar estimates as in \cite[Section 4.6]{KS2020},
	\begin{equation}\label{Equation EV1}
	\mathbb{E}[\|\mathbf{v}\|_{H^{-\gamma/2}}^2\exp((\nu-L) t-C_1 \nu^{-1}\lambda_N^{-\gamma/2+1/3}\int_{0}^{t}\|\xi\|_{H^{-\gamma/2}}^2ds)]\leq\|x-y\|_{H^{-\gamma/2}}^2,\quad t\geq 0,
	\end{equation}
	where $\xi=\nabla^{\bot}\cdot\mathbf{\mathbf{u}}$  and $C_1>0$ is a universal constant. Furthermore, we have the following lemma.
	
	\begin{lemma}\label{EV Lemma1}
		Assume that $\mathbf{EV_1}$-$\mathbf{EV_3}$. Then\\
		$(1)$ $\;\mathbf{v}(t)$ converges to $0$ almost surely as $t\rightarrow\infty;$\\
		$(2)$ 
		\begin{equation*}
		d_{TV}(P_t(y,\cdot),{\rm Law}(\tilde{\mathbf{u}}(t)))\leq c_1\|x-y\|_{H^{-\gamma/2}}^{c_2}\exp(c_3\|x\|_{H^{1-\gamma/2}}^2),\; t\geq 0.
		\end{equation*}
		where $c_1,c_2,c_3$ are positive constants, independent of $t$.
		
	\end{lemma}

	Hence  the $d$-eventual continuity can be argued as in Theorem \ref{E Thm}.

	\begin{proposition}\label{EV Prop2}
		Assume $\mathbf{EV_1}$ - $\mathbf{EV_3}$.
		Then the Markov semigroup $\{P_t\}_{t\geq 0}$ is $d$-eventually continuous on $V$. 
	\end{proposition}

	Finally, we show the uniform irreducibility. Again, we prove it under the uniform nondegenerate assumption.
	
	\begin{itemize}
		\item[$\mathbf{EV_2'}$] The function $\sigma$ is uniformly nondegenerate, i.e., there exists a constant $C_0$ such that for all $\mathbf{w}\in V$
		\begin{equation*}
		\sup\limits_{\mathbf{u}\in V}\|\sigma^{-1}(\mathbf{u})(\mathbf{w})\|\leq C_0\|\mathbf{w}\|.
		\end{equation*}
	\end{itemize}
	
	\paragraph{Uniform irreducibility} Still, we apply similar arguments as in \cite[Lemma 3.8]{HM2011} to show the uniform irreducibility. 
	\begin{proposition}\label{EV Prop UI}
	Assume that $\mathbf{EV_1}$ and $\mathbf{EV_2'}$, then $\{P_t\}_{t\geq 0}$ is unformly irreducibile in the $H^{-\gamma/2}$-topology.
	\end{proposition}

	Applying Theorem \ref{Thm 5}, we derive the following theorem.
	\begin{theorem} 
		Assume that $\mathbf{EV_1}$ and $\mathbf{EV_2'}$.
		Then  $\{P_t\}_{t\geq 0}$ has unique invariant meausre $\mu$ on $V$, such that $P_t(x,\cdot)$ converges weakly to $\mu$ in the $H^{-\gamma/2}$-topology.
	\end{theorem}

     \section{Proofs}\label{Sec 4}

	 \subsection{Proof of Lemma  \ref{NS Lemma1}}

	We first provide an useful lemma.
	\begin{lemma}\label{Lemma M}
	    Let $M$ be a continuous  local martingale with $M_0=0$. Define
		\begin{equation*}
		\Xi_\kappa:=\sup\limits_{t\geq 0}(M_t-\kappa\langle M\rangle_t).
		\end{equation*} Then 
		\begin{equation*}
		\mathbb{P}(\Xi_\kappa\geq R)\leq \exp(-2\kappa R),\quad\forall R\geq 0,
		\end{equation*}
	\end{lemma}
	
	\begin{proof}[Proof.] 
		By the Dambis–Dubins–Schwarz theorem, there exists (maybe on an extended probability space) a Brownian motion $\widehat{B}$ such that $M_t=\widehat{B}_{\langle M\rangle_t}$ for $t\geq 0$. 
		Let $\kappa>0$ be a sufficiently small parameter to be chosen later. Hence
		\begin{equation*}
		    \Xi_\kappa\leq \sup\limits_{t\geq 0}(\widehat{B}_t-\kappa t),
		\end{equation*}
		we have
		\begin{equation*}
		\mathbb{P}(\Xi_\kappa\geq R)\leq \exp(-2\kappa R),\quad\forall R\geq 0.
		\end{equation*}
	\end{proof}
	
	\paragraph{Proof of Lemma \ref{NS Lemma1}:}
 
	\begin{proof}[Proof.] 
		(1)\  Applying the It\^o lemma to (\ref{Equation NS}) we find that
		\begin{equation}\label{Equation NS2}
		d|\mathbf{u}|^2+2\nu\|\mathbf{u}\|^2dt=|\sigma(\mathbf{u})|^2dt+2\langle\sigma(\mathbf{u}),\mathbf{u}\rangle dW_t.
		\end{equation}\par 
		Let $M_t:=\int_{0}^{t}2\langle\sigma(\mathbf{u}),\mathbf{u}\rangle dW_s $ be a continuous local martingale. (\ref{Equation NS2}) implies that
		\begin{equation*}
		|\mathbf{u}|^2+2\nu\int_{0}^{t}\|\mathbf{u}\|^2ds\leq |x|^2+B_0t+M_t,
		\end{equation*}
		and 
		\begin{equation*}
		d\langle M\rangle_t=4|\langle\sigma(\mathbf{u}),\mathbf{u}\rangle|^2dt\leq 4B_0|\mathbf{u}|^2dt\leq 4B_0\|\mathbf{u}\|^2dt.
		\end{equation*}	\par 
		By Lemma \ref{Lemma M}, we have
		\begin{equation}\label{Equation NS5.2}
		\mathbb{P}(\Xi_\kappa\geq R)\leq \exp(-2\kappa R),\quad\forall R\geq 0,
		\end{equation}
		and
		\begin{equation}\label{Equation NS5.3}
		|\mathbf{u}_t|^2+(2\nu-4\kappa B_0)\int_{0}^{t}\|\mathbf{u}\|^2ds\leq |x|^2+B_0t+\Xi_\kappa.
		\end{equation}
		Let $\kappa=\frac{\nu}{4B_0}$ for simplicity, we obtain that
		\begin{equation*}
		\nu\int_{0}^{t}\|\mathbf{u}\|^2ds\leq |x|^2+B_0t+\Xi_\kappa.
		\end{equation*}
		Hence we obtain that 
		\begin{equation*}
		\limsup\limits_{t\rightarrow\infty}	\frac{\nu}{t}\int_{0}^{t}\|\mathbf{u}\|^2ds\leq B_0,\quad \text{ almost surely},
		\end{equation*}
	which in turn implies that $\exp\{(\nu\lambda_N-L)t-\frac{C_D^2}{\nu}\int_{0}^{t}\|\mathbf{u}\|^2ds\}$ diverges to infinity almost surely as $t\rightarrow\infty$ under Assumption $\mathbf{H_3}.$  Then it follows from (\ref{Equation NS9.1}) that $\mathbf{v}_t$ converges to $0$ almost surely.\\		
		(2)\  We can write (\ref{Equation NS6}) in the form (\ref{Equation NS}) with $x$ changed to $y$, and with $W_t$  replaced by
		\begin{equation}\label{Equation NS10}
		\widetilde{W}_t:=W_t+\int_{0}^{t}\beta_sds,\quad\beta_t:=\frac{\nu\lambda_N}{2}\sigma^{-1}(\tilde{\mathbf{u}}^y_t)P_N(\mathbf{u}_t^x-\tilde{\mathbf{u}}_t^y).
		\end{equation}
		By $\mathbf{H_2}$, the pseudo-inverse operator $\sigma^{-1}(\tilde{\mathbf{u}}^y_t)$ is well defined and bounded; thus there exists a constant $C_0>0$ such that for all $t\geq0$
		\begin{equation*}
		|\beta_t|\leq \nu\lambda_NC_0 |\mathbf{v}_t|.
		\end{equation*}

		Further, as  $\mathbf{u}^y_t$ is the strong solution to equation (\ref{Equation NS}) with the initial value y, we can regard $\mathbf{u}^y_t$ as an image of the driving noise under some measurable mapping
		
		\begin{equation*}
		\Phi_t^y:C({0,t};\mathbb{R}^m)\rightarrow H.
		\end{equation*}
		In other words, $\mathbf{u}^y_t= \Phi_t^y (W_{[0,t]}),$ where $W_{[0,t]}:=\{W_s:s\in[0,t]\},\;t \geq 0.$ It follows from the Girsanov theorem that $\text{Law}(\widetilde{W}_{[0,t]})$ is absolutely continuous with respect to $\text{Law}(W_{[0,t]}).$ Therefore, by the uniqueness of the solution, we have $\tilde{\mathbf{u}}^y_t=\Phi_t^y (\widetilde{W}_{[0,t]}).$\par 
		Thus, by \cite[Theorem A.5]{BKS2020}, we  derive that for any $\delta\in(0,1]$ there exists $C_\delta > 0$ such that for any $t\geq 0$
		\begin{equation*}
		\begin{aligned}
		d_{TV}(P_t(y,\cdot),\text{Law}(\tilde{\mathbf{u}}^y_t))&=d_{TV}(\text{Law}({\mathbf{u}}^y_t),\text{Law}(\tilde{\mathbf{u}}^y_t))\\
		&=d_{TV}(\text{Law}( \Phi_t^y (W_{[0,t]}),\text{Law}(\Phi_t^y (\widetilde{W}_{[0,t]})))\\
		&\leq d_{TV}(\text{Law}(W_{[0,t]}),\text{Law}(\widetilde{W}_{[0,t]}))\\
		&\leq C_\delta(\mathbb{E}(\int_{0}^{t}|\beta_s|^2ds)^\delta)^{1/(1+\delta)}\\
		&\leq  C_\delta(\nu\lambda_N C_0)^{\frac{2\delta}{1+\delta}}(\mathbb{E}(\int_{0}^{t}|\mathbf{v}_s|^2ds)^\delta)^{1/(1+\delta)}.
		\end{aligned}
		\end{equation*}
		
		We need to estimate $\mathbb{E}(\int_{0}^{t}|\mathbf{v}_s|^2ds)^\delta.$ We shall split it into two parts, and estimate each term as a whole, seperately.	For brevity, denote 
		\begin{equation*}
		h_p(t)=(\nu\lambda_N-\frac{p+1}{2}L)t-\frac{C_{D}^2}{\nu}\int_{0}^{t}\|\mathbf{u}_s\|^2ds.
		\end{equation*}\par 
		Fix $\epsilon>0,$ by H\"{o}lder ineqaulity
		\begin{equation}\label{Equation NS14.1}
		\begin{aligned}
		\mathbb{E}(\int_{0}^{t}|\mathbf{v}_s|^2ds)^\delta&=\mathbb{E}(\int_{0}^{t}|\mathbf{v}_s|^2\exp(h_p-\epsilon s)\exp(-h_p+\epsilon s)ds)^{\delta}\\
		&\leq \mathbb{E}[(\int_{0}^{t}|\mathbf{v}_s|^{2p}\exp[p(h_p-\epsilon s)]ds)^{\frac{\delta}{p}} (\int_{0}^{t}\exp[q(-h_p+\epsilon s)]ds)^{\frac{\delta}{q}}]\\
		&\leq (\int_{0}^{t}\mathbb{E}|\mathbf{v}_s|^{2p}\exp[p(h_p-\epsilon s)]ds)^{\frac{\delta}{p}}(\mathbb{E}(\int_{0}^{t}\exp[q(-h_p+\epsilon s)]ds)^{\frac{(p-1)\delta}{p-\delta}})^{\frac{p-\delta}{p}},
		\end{aligned}
		\end{equation}
		where $p>1,1/p+1/q=1.$\par

		The estimation of the first term is similar to (\ref{Equation NS9.1}). For $p>1,$ we use It\^o lemma again to obtain
		\begin{equation*}
		d|\mathbf{v}|^{2p}=p|\mathbf{v}|^{2p-2}d|\mathbf{v}|^2+\frac{p(p-1)}{2}|\langle\sigma(\mathbf{u}^x)-\sigma(\tilde{\mathbf{u}}^y),\mathbf{v}\rangle|^2dt,
		\end{equation*}
		hence 
		\begin{equation*}
		d|\mathbf{v}|^{2p}+p(\nu\lambda_N-\frac{p+1}{2}L)|\mathbf{v}|^{2p}dt\leq\frac{pC_{D}^2}{\nu}|\mathbf{v}|^{2p}\|\mathbf{u}^x\|^2dt+p|\mathbf{v}|^{2p-2}\langle\sigma(\mathbf{u}^x)-\sigma(\tilde{\mathbf{u}}^y),\mathbf{v}\rangle dW_t,
		\end{equation*}
		which implies that
		\begin{equation}\label{Equation NS9.2}
		\mathbb{E}[|\mathbf{v}|^2\exp((\nu\lambda_N-\frac{p+1}{2}L)t-\frac{C_{D}^2}{\nu}\int_{0}^{t}\|\mathbf{u}\|^2ds)]^p\leq |x-y|^{2p}.
		\end{equation}\par

		Multiplying (\ref{Equation NS9.2}) with $ \exp(-p\epsilon t)$, then integrating it form $0$ to $t$, it follows that
		\begin{equation*}
		\int_{0}^{t}\mathbb{E}|\mathbf{v}|^{2p}\exp[p(h_p-\epsilon s)]ds\leq (p\epsilon)^{-1}|x-y|^{2p},\quad\forall t\geq 0.
		\end{equation*}
		To control $\mathbb{E}(\int_{0}^{t}\exp[q(-h_p+\epsilon s)]ds)^{\frac{(p-1)\delta}{p-\delta}}$, we need some estimates of $\int_{0}^{t}\|\mathbf{u}\|^2ds.$
		By (\ref{Equation NS5.3})
		\begin{equation*}
		-h_p(t)\leq \frac{ C_{D}^2}{\nu^2}|x|^2+(-\nu\lambda_N+\frac{p+1}{2}L+\frac{C_{D}^2}{\nu^2}B_0)t+\frac{ C_{D}^2}{\nu^2}\Xi_\kappa.
		\end{equation*}

		Denoted by $K:=\nu\lambda_N-L-\frac{C_D^2}{\nu^2}B_0>0,\;k=\frac{C_D^2}{\nu^2}$, fixing $p=1+\frac{K}{L},\;\epsilon=\frac{K}{4}$, we have
		\begin{equation}\label{Equation NS14.2}
		\begin{aligned}
		\mathbb{E}(\int_{0}^{t}\exp[q(-h_p(s)+\epsilon s)]ds)^{\frac{(p-1)\delta}{p-\delta}}&\leq \exp(\frac{kp\delta}{p-\delta}|x|^2)\int_{0}^{t}e^{-\frac{(Kp\delta }{4(p-\delta)}s}ds	\mathbb{E}\exp(\frac{kp\delta}{p-\delta}\Xi_\kappa).\\
		&\leq \frac{Kp\delta }{4(p-\delta)}\exp(\frac{kp\delta}{p-\delta}|x|^2)	\mathbb{E}\exp(\frac{kp\delta}{p-\delta}\Xi_\kappa).
		\end{aligned}
		\end{equation}\par

		Now we can choose $\delta$ such that $\frac{kp\delta}{p-\delta}=\kappa,$ then by (\ref{Equation NS5.2}) and (\ref{Equation NS14.2}),  we have
		\begin{equation*}
		\mathbb{E}(\int_{0}^{t}\exp[q(-h_p+\epsilon s)]ds)^{\frac{(p-1)\delta}{p-\delta}}\leq \frac{K\nu }{8kB_0}\exp(\frac{\nu}{4B_0}|x|^2).
		\end{equation*} \par 
		Combining these results, we have
		\begin{equation*}
		\begin{aligned}
		d_{TV}(P_t(y,\cdot),\text{Law}(\tilde{\mathbf{u}}^y_t))&\leq  C_\delta(\nu\lambda_N C_0)^{\frac{2\delta}{1+\delta}}(\mathbb{E}(\int_{0}^{t}|\mathbf{v}_s|^2ds)^\delta)^{1/(1+\delta)}\\
		&\leq C(\int_{0}^{t}\mathbb{E}|\mathbf{v}|^{2p}\exp[p(h_p-\epsilon s)]ds)^{\frac{\delta}{p(1+\delta)}}\\
            &\quad \cdot(\mathbb{E}(\int_{0}^{t}\exp[p(-h_p+\epsilon s)]ds)^{\frac{(p-1)\delta}{p-\delta}})^{\frac{p-\delta}{p(1+\delta)}}\\
		&\leq C|x-y|^{\frac{2\delta}{1+\delta}}\exp(\frac{\nu}{4B_0}|x|^2),
		\end{aligned}
		\end{equation*}
		where $C=C(\nu,\lambda_N,B_0,L,C_0,C_D),\;\delta=\frac{p\nu}{\nu+4pkB_0},\;p=1+\frac{K}{L},\;K=\nu\lambda_N-L-\frac{C_D^2}{\nu^2}B_0,\;k=\frac{C_D^2}{\nu^2}.$
	\end{proof}

	\subsection{Proof of Lemma  \ref{Lemma L1}}
	
	\begin{proof}[Proof.] 
	To prove Lemma \ref{Lemma L1}, we use the Fourier coefficients of $\mathbf{u}_t^x$. Recall that
	\begin{equation*}
	\widehat{\mathbf{u}^x_t}(k)=(2\pi)^{-d}\int_{\mathbb{T}^d}\mathbf{u}^x_t(\xi)e^{-i\xi\cdot k}d\xi,\;k\in\mathbb{Z}_*^d,
	\end{equation*}
	and we denote $\widehat{\mathbf{u}^x_t}(k):=a_k-ib_k$ for brievity, where
	\begin{equation*}
	a_k=(2\pi)^{-d}\int_{\mathbb{T}^d}\mathbf{u}^x_t(\xi)\cos (\xi\cdot k)d\xi,,\quad b_k=(2\pi)^{-d}\int_{\mathbb{T}^d}\mathbf{u}^x_t(\xi)\sin (\xi\cdot k)d\xi,\;k\in\mathbb{Z}_*^d.
	\end{equation*}
	Then
	\begin{equation*}
	A\mathbf{u}_t^x(\xi)=\sum_{k\in\mathbb{Z}_*^d}-\gamma(k)(a_k\cos (\xi\cdot k)+b_k\sin (\xi\cdot k)),
	\end{equation*}
	and
	\begin{equation*}
	B(\mathbf{u}_t^x,\mathbf{u}_t^x)(\xi)=\sum_{k\in\mathbb{Z}_*^d}\mathbf{u}_t^x(0)\cdot k(b_k\cos (\xi\cdot k)-a_k\sin (\xi\cdot k)).
	\end{equation*}
	we denote $Q_{k,1}^{1/2}(\mathbf{u}_t^x)dW_t$ and $Q_{k,2}^{1/2}(\mathbf{u}_t^x)dW_t$ as the component of $Q(\mathbf{u}_t^x)^{1/2}dW_t$ on modes $\cos (\xi\cdot x)$ and $\sin (\xi\cdot x)$, respsectively. Hence we obtain (\ref{Equation L0}) in the Fourier coeficients form as 
	\begin{equation*}
	\begin{aligned}
	da_k&=-\gamma(k)a_kdt+\mathbf{u}_t^x(0)\cdot k b_kdt+Q_{k,1}^{1/2}(\mathbf{u}_t^x)dW_t,\\
	db_k&=-\gamma(k)b_kdt-\mathbf{u}_t^x(0)\cdot k a_kdt+Q_{k,2}^{1/2}(\mathbf{u}_t^x)dW_t.
	\end{aligned}
	\end{equation*}
	Note that for $|k|>M,$ $q_k(\cdot)\equiv 1,$ we write $Q_{k,i}^{1/2}=Q_{k,i}^{1/2}(\mathbf{u}_t^x)$ for $i=1,2,\;|k|>M.$\par 
	
	Applying It\^o lemma, we have
	\begin{equation*}
	\begin{aligned}
	d(|a_k|^2+|b_k|^2)&=-2\gamma(k)(|a_k|^2+|b_k|^2)dt+q_k(\mathbf{u}_t^x)\gamma(k)\text{Tr }\mathcal{E}(k)dt+2\langle a_k,Q_{k,1}^{1/2}dW_t\rangle\\
        &\quad+2\langle b_k,Q_{k,2}^{1/2}dW_t\rangle \\
	&\leq -2\gamma(k)(|a_k|^2+|b_k|^2)dt+\beta\gamma(k)\text{Tr }\mathcal{E}(k)dt+2\langle a_k,Q_{k,1}^{1/2}dW_t\rangle\\
        &\quad+2\langle b_k,Q_{k,2}^{1/2}dW_t\rangle
	\end{aligned}
	\end{equation*}
	Hence by Gronwall inequality,
	\begin{equation*}
		\mathbb{E}|k|^{2m}(|a_k|^2+|b_k|^2)\leq e^{-2\gamma(k)t}|k|^{2m}|x_k|^2+\beta|k|^{2m}\text{Tr }\mathcal{E}(k),
	\end{equation*}
	where $ x_k=(2\pi)^{-d}\int_{\mathbb{T}^d}xe^{-i\xi\cdot k}d\xi,\;k\in\mathbb{Z}_*^d$.\par 	
	Considering assumptions $\mathbf{H_1},\mathbf{H_2}$, finally, we obtain
	\begin{equation*}
	\begin{aligned}
		\mathbb{E}|\mathbf{u}_t|^2&=\sum_{k\in\mathbb{Z}_*^d}\mathbb{E}|k|^{2m}(|a_k|^2+|b_k|^2)\leq\sum_{k\in\mathbb{Z}_*^d} e^{-2\gamma(k)t}|k|^{2m}|x_k|^2+\beta\sum_{k\in\mathbb{Z}_*^d}|k|^{2m}\text{Tr }\mathcal{E}(k)\\
		&\leq e^{-2\gamma_*t}|x|^2+\beta\gamma_*^{-\alpha}|||\mathcal{E}|||.
	\end{aligned}
	\end{equation*} 

	\end{proof}
	
  \subsection{Proof of Lemma  \ref{Lemma L2}}
 
	\begin{proof}[Proof.] 
	The proof resembles that of \cite[Lemma 2.4]{SS2002} and \cite[Lemma 3.8]{HM2011}. Fix $x\in B(0,R),$ let $\mathbf{w}$ be the solution of equation
	\begin{equation*}
	\left\{	\begin{aligned}
		&d\mathbf{w}_t=A\mathbf{w}_tdt +B(\mathbf{w}_t,\mathbf{w}_t)dt,\\
		&\mathbf{w}_0=x.
		\end{aligned}\right.
	\end{equation*}
	Differentiating $|\mathbf{w}_t|^2$ over $t$, it is easy to show that 
	\begin{equation*}
		|\mathbf{w}_t|^2\leq e^{-2\gamma_*t}|x|^2 \leq e^{-2\gamma_*t}R^2.
	\end{equation*}
	Hence, we choose $T=T(R)>0$ sufficiently large such that $|\mathbf{w}_t|<\epsilon/2$. Furthermore, same as the proof in \cite[Section 5.1]{KPS2010}, there exists some $K>0$ such that
	\begin{equation*}
		\|\mathbf{w}_t\|^2=\|S(t)x\|^2.
	\end{equation*}
	While $S(t)$ is bounded form $\mathcal{X}$ to $\mathcal{V}$ by \cite[Lemma 6]{KPS2010}, we obtain
	\begin{equation*}
		\sup\limits_{t\geq 0}\|\mathbf{w}_t\|=\sup\limits_{t\geq 0}\|S(t)x\|\leq K<\infty.
	\end{equation*}\par 
	Define 
	\begin{equation*}
		D(t):=|\mathbf{u}_t-\mathbf{w}_t-\frac{\epsilon}{2}e|^2-(\frac{\epsilon}{2})^2,
	\end{equation*}
	where $e(x):=e^{ij\cdot x}$ for some $j\in\mathbb{Z}_*^d$. Then 
	\begin{equation*}
	\begin{aligned}
		dD(t)&=2\langle\mathbf{u}_t-\mathbf{w}_t-\frac{\epsilon}{2}e,A(\mathbf{u}_t-\mathbf{w}_t)+B(\mathbf{u}_t,\mathbf{u}_t)-B(\mathbf{w}_t,\mathbf{w}_t)\rangle dt\\
		&\quad+\sum_{k\in\mathbb{Z}_*^d}q_k(\mathbf{u}_t)\gamma(k)|k|^{2m}\text{Tr }\mathcal{E}(k)dt+2\langle \mathbf{u}_t-\mathbf{w}_t-\frac{\epsilon}{2}e, Q^{1/2}(\mathbf{u}_t)dW_t\rangle ,
	\end{aligned}
	\end{equation*}
	and $D(0)=0$. Define 
	\begin{equation*}
		\tau:=\inf\{t\geq 0:|D(t)|>(\frac{\epsilon}{4})^2\}.
	\end{equation*}
	For $t\in [0,\tau)$,
	\begin{equation*}
		\frac{3\epsilon^2}{16}<|\mathbf{u}_t-\mathbf{w}_t-\frac{\epsilon}{2}e|^2< \frac{5\epsilon^2}{16},\quad|\mathbf{u}_t-\mathbf{w}_t|^2< \frac{9\epsilon^2}{16},\quad |\mathbf{u}_t|^2<|\mathbf{w}_t|^2+ \frac{9\epsilon^2}{16}\leq R^2+ \frac{9\epsilon^2}{16}.
	\end{equation*}
	Then
	\begin{equation*}
		\langle\mathbf{u}_t-\mathbf{w}_t-\frac{\epsilon}{2}e,A(\mathbf{u}_t-\mathbf{w}_t)\rangle\leq \frac{\epsilon}{2}\gamma(j)|\mathbf{u}_t-\mathbf{w}_t|,
	\end{equation*}
	 and that
	\begin{equation*}
	\begin{aligned}
	&\langle\mathbf{u}_t-\mathbf{w}_t-\frac{\epsilon}{2}e,B(\mathbf{u}_t,\mathbf{u}_t)-B(\mathbf{w}_t,\mathbf{w}_t)\rangle\\
    &=\langle\mathbf{u}_t-\mathbf{w}_t-\frac{\epsilon}{2}e,B(\mathbf{u}_t,\mathbf{u}_t-\mathbf{w}_t)\rangle+\langle\mathbf{u}_t-\mathbf{w}_t-\frac{\epsilon}{2}e,B(\mathbf{u}_t-\mathbf{w}_t,\mathbf{w}_t)\rangle\\
	&=-\frac{\epsilon}{2}	\langle e,B(\mathbf{u}_t,\mathbf{u}_t-\mathbf{w}_t)\rangle+\langle\mathbf{u}_t-\mathbf{w}_t-\frac{\epsilon}{2}e,B(\mathbf{u}_t-\mathbf{w}_t,\mathbf{w}_t)\rangle\\
	&\leq C_1|\mathbf{u}_t|\;|\mathbf{u}_t-\mathbf{w}_t|+C_1|\mathbf{u}_t-\mathbf{w}_t-\frac{\epsilon}{2}e|\,|\mathbf{u}_t-\mathbf{w}_t|\,\|\mathbf{w}_t\|,
	\end{aligned}
	\end{equation*} 
	for some constant $C_1>0.$ Furthermore, letting $M_t:=\int_{0}^{t}2\langle \mathbf{u}_s-\mathbf{w}_s-\frac{\epsilon}{2}e, Q^{1/2}(\mathbf{u}_s)dW_s\rangle $  be the local martingale part, we infer that
	\begin{equation*}
		\frac{d}{dt}\langle M\rangle_t=4|Q^{1/2}(\mathbf{u}_t)(\mathbf{u}_t-\mathbf{w}_t-\frac{\epsilon}{2}e)|^2\leq \beta C_2|||\mathcal{E}|||\,|\mathbf{u}_t-\mathbf{w}_t-\frac{\epsilon}{2}e|^2,
	\end{equation*}
	 for some constant $C_2>0$. On ther other hand, let $N\in\mathbb{N}$ be sufficently large such that
	 \begin{equation*}
	 	\sup\limits_{0\leq t\leq \tau\wedge T}|P_N(\mathbf{u}_t-\mathbf{w}_t-\frac{\epsilon}{2}e)|^2>	\frac{\epsilon^2}{16}.
	 \end{equation*}
	 By $\mathbf{H_3}$, the inverse operator $Q^{-1/2}(\mathbf{u}_t)$ is well defined and bounded, thus there exists a constant $C_0>0$ such that  $|Q^{-1/2}(\mathbf{u}_t)(P_Ny)|\leq C_0|P_Ny|$ for $y\in\mathcal{X}$. Consequently,
	\begin{equation*}
	\begin{aligned}
		\frac{d}{dt}\langle M\rangle_t&=4|Q^{1/2}(\mathbf{u}_t)(\mathbf{u}_t-\mathbf{w}_t-\frac{\epsilon}{2}e)|^2\geq 4|Q^{1/2}(\mathbf{u}_t)(P_N(\mathbf{u}_t-\mathbf{w}_t-\frac{\epsilon}{2}e))|^2\\
		&\geq 4 C_0|P_N(\mathbf{u}_t-\mathbf{w}_t-\frac{\epsilon}{2}e)|^2>\frac{C_0\epsilon^2}{4}>0.
	\end{aligned}
	\end{equation*}\par 	
	Let now $\widetilde{W}$ be a Wiener	process that is independent of $W$ and define
	\begin{equation*}
		Y(y):=D(t\wedge \tau)+(\widetilde{W}_t-\widetilde{W}_\tau)\mathbf{1}_{\{t\geq\tau\}}.
	\end{equation*}
	Then $Y(t)$ is a semimartingale with $Y(0)=0$, which fulfills the conditions of \cite[Lemma I.8.3]{B1998} . Therefore, there exists $p>0$ such that for all $x\in B(0,R)$,
	\begin{equation*}
		\mathbb{P}(\sup\limits_{0\leq t\leq T}|Y(t)|\leq (\frac{\epsilon}{4})^2)\geq p.
	\end{equation*}\par 
	Then
	\begin{equation*}
		P_T(x,B(0,\epsilon))=\mathbb{P}(|\mathbf{u}_T|<\epsilon)\geq \mathbb{P}(\tau> T)\geq 	\mathbb{P}(\sup\limits_{0\leq t\leq T}|Y(t)|\leq (\frac{\epsilon}{4})^2)\geq p
	\end{equation*}
	for all $x\in B(0,R),$ completing the proof.
	\end{proof}

     \subsection{Proof of Lemma  \ref{Lemma L3}}
	
		\begin{proof}[Proof.] 
	(1) By (\ref{Equation L8}), again, applying It\^o lemma and $\mathbf{H_2}$, we obtain
	\begin{equation*}
	d|\mathbf{v}_t|^8\leq -4\gamma_*|\mathbf{v}_t|^8dt+(28K^2\|Q\|-\lambda)|P_N\mathbf{v}_t|^2|\mathbf{v}|^6dt+4\tilde{C} h(\|Q_N\mathbf{u}_t^x\|)|\mathbf{v}_t|^8dt+8|\mathbf{v}_t|^6g_t.
	\end{equation*}
	Let $\lambda = 28K^2\|Q\|$ and define stopping times
	\begin{equation*}
		\tau_n:=\inf\{t\geq 0:|\mathbf{v}_t|>n\}.
	\end{equation*}	
	Using Gronwall inequality, we obtain that
	\begin{equation*}
		\mathbb{E}|\mathbf{v}_{t\wedge \tau_n}|^8\exp\{4\gamma_*t\wedge\tau_n-4\tilde{C}\int_{0}^{t\wedge\tau_n}h(\|Q_N\mathbf{u}_s^x\|)ds\}\leq |x-y|^8.
	\end{equation*}
	Denote $\varphi_n(t):=\exp\{2\gamma_*t\wedge\tau_n-2\tilde{C}\int_{0}^{t\wedge\tau_n}h(\|Q_N\mathbf{u}_s^x\|)ds\}$. By Cauchy-Schwarz inequality, we have
	\begin{align*}
			\mathbb{E}|\mathbf{v}_{t\wedge \tau_n}|^4&=\mathbb{E}|\mathbf{v}_{t\wedge \tau_n}|^4\varphi_n(t)\varphi_n^{-1}(t)\\
			&\leq (\mathbb{E}|\mathbf{v}_{t\wedge \tau_n}|^8\varphi_n^2(t))^{1/2}(\mathbb{E}\varphi_n^{-2}(t))^{1/2}\\
			&\leq |x-y|^4(\mathbb{E}\exp\{-4\gamma_*t\wedge\tau_n+4\tilde{C}\int_{0}^{t\wedge\tau_n}h(\|Q_N\mathbf{u}_s^x\|)ds\})^{1/2}\\
			&\leq  |x-y|^4(\mathbb{E}\exp\{-8\gamma_*t\wedge\tau_n\})^{1/4}(\mathbb{E}\exp\{8\tilde{C}\int_{0}^{t\wedge\tau_n}h(\|Q_N\mathbf{u}_s^x\|)ds\})^{1/4}
	\end{align*}
	Letting $n$ increse to infinity, it implies that
	\begin{equation*}
		\mathbb{E}|\mathbf{v}_t|^4\leq |x-y|^4 e^{-2\gamma_*t}(\mathbb{E}\exp\{8\tilde{C}\int_{0}^{t}h(\|Q_N\mathbf{u}_s^x\|)ds\})^{1/4}.
	\end{equation*} 
	The proof of part (1) of the lemma will be completed as soon as we can show that
	there exists an $N_0>M$ such that, for all $N\geq N_0,$
	\begin{equation*}
		\mathbb{E}|\exp\{8\tilde{C}\int_{0}^{t}h(\|Q_N\mathbf{u}_s^x\|)ds\}\leq C_1e^{\kappa t},
	\end{equation*}
	for some $C_1>0$ and $0<\kappa<8\gamma_*.$\par 

	 For $z=(x_k,y_k),\,k\in\mathbb{Z}_*^d,\,|k|\geq N,\,x_k,y_k\in\mathbb{R},$ define
	\begin{equation*}
		H(z):=(1+\sum_{k\in\mathbb{Z}_*^d,|k|\geq N}\gamma(k)^{-1}|k|^{2(m+1)}(x_k^2+y_k^2))^{1/2}.
	\end{equation*}
	Then
	\begin{equation*}
		H'(z)=H^{-1}(z)(\gamma(k)^{-1}|k|^{2(m+1)}x_k,\gamma(k)^{-1}|k|^{2(m+1)}y_k)_{k\in\mathbb{Z}_*^d,\,|k|\geq N}
	\end{equation*} and
	\begin{equation*}
	\begin{aligned}
		H''(z)&=H^{-1}(z)\text{ diag }(\gamma(k)^{-1}|k|^{2(m+1)}x_k,\gamma(k)^{-1}|k|^{2(m+1)}y_k)_{k\in\mathbb{Z}_*^d,\,|k|\geq N} \\
	&\quad-H^{-3}(z)(\gamma(k)^{-1}|k|^{2(m+1)}x_k,\gamma(k)^{-1}|k|^{2(m+1)}y_k)\\
 &\quad\quad\otimes(\gamma(k)^{-1}|k|^{2(m+1)}x_k,\gamma(k)^{-1}|k|^{2(m+1)}y_k)_{k\in\mathbb{Z}_*^d,\,|k|\geq N}.
	\end{aligned}
	\end{equation*}\par 
	Let $u:=(a_k,b_k)_{k\in\mathbb{Z}_*^d,\, |k|\geq N}$. Applying It\^o lemma, we have
	\begin{equation*}
	\begin{aligned}
		0\leq H(u)&=1-\int_{0}^{t}H^{-1}(u)\sum_{k\in\mathbb{Z}_*^d,|k|\geq N}|k|^{2(m+1)}(a_k^2+b_k^2)ds\\
		&\quad +\frac{1}{2}\int_{0}^{t}H^{-1}(u)\sum_{k\in\mathbb{Z}_*^d,|k|\geq N}|k|^{2(m+1)}\text{Tr }\mathcal{E}(k)ds\\
		&\quad-\frac{1}{2}\int_{0}^{t}H^{-3}(u)\sum_{k\in\mathbb{Z}_*^d,|k|\geq N}\gamma(k)^{-1}|k|^{4(m+1)}\text{Tr }\mathcal{E}(k)(a_k^2+b_k^2)ds\\
		&\quad+\int_{0}^{t}H^{-1}(u)\sum_{k\in\mathbb{Z}_*^d,|k|\geq N}\gamma(k)^{-1}|k|^{2(m+1)}(\langle a_k,Q_{k,1}^{1/2}dW_t\rangle+\langle b_k,Q_{k,2}^{1/2}dW_t\rangle).
	\end{aligned}
	\end{equation*}
	By definition of $h(\|Q_N\mathbf{u}_t^x\|)$, we obtain 
	\begin{equation*}
	\begin{aligned}
		8\tilde{C}\int_{0}^{t}h(\|Q_N\mathbf{u}_s^x\|)ds&\leq	 8\tilde{C}\int_{0}^{t}H^{-1}(u)\sum_{k\in\mathbb{Z}_*^d,|k|\geq N}|k|^{2(m+1)}(a_k^2+b_k^2)ds\\
		&\leq 	8\tilde{C}+	4\tilde{C}\int_{0}^{t}H^{-1}(u)\sum_{k\in\mathbb{Z}_*^d,|k|\geq N}|k|^{2(m+1)}\text{Tr }\mathcal{E}(k)ds\\
		&\quad+	8\tilde{C}\int_{0}^{t}H^{-1}(u)\sum_{k\in\mathbb{Z}_*^d,|k|\geq N}\gamma(k)^{-1}|k|^{2(m+1)}(\langle a_k,Q_{k,1}^{1/2}dW_t\rangle\\
        &\quad+\langle b_k,Q_{k,2}^{1/2}dW_t\rangle)\\
		&\leq 8\tilde{C}+\mathcal{M}_t+\mathcal{R}_t,	
	\end{aligned}
	\end{equation*} where
	\begin{equation*}
	\begin{aligned}
		\mathcal{M}_t&:=8\tilde{C}\int_{0}^{t}H^{-1}(u)\sum_{k\in\mathbb{Z}_*^d,|k|\geq N}\gamma(k)^{-1}|k|^{2(m+1)}(\langle a_k,Q_{k,1}^{1/2}dW_t\rangle+\langle b_k,Q_{k,2}^{1/2}dW_t\rangle)\\
		&\quad-128\tilde{C}\int_{0}^{t}H^{-2}(u)\sum_{k\in\mathbb{Z}_*^d,|k|\geq N}\gamma(k)^{-1}|k|^{4(m+1)}\text{Tr }\mathcal{E}(k)ds
	\end{aligned}
	\end{equation*}
	and 
	\begin{equation*}
	\begin{aligned}
		\mathcal{R}_t&:=4\tilde{C}\int_{0}^{t}H^{-1}(u)\sum_{k\in\mathbb{Z}_*^d,|k|\geq N}|k|^{2(m+1)}\text{Tr }\mathcal{E}(k)ds\\
        &\quad+128\tilde{C}\int_{0}^{t}H^{-2}(u)\sum_{k\in\mathbb{Z}_*^d,|k|\geq N}\gamma(k)^{-1}|k|^{4(m+1)}\text{Tr }\mathcal{E}(k)ds\\
		&\leq 128\tilde{C}t\sum_{k\in\mathbb{Z}_*^d,|k|\geq N}(|k|^{2(m+1)}\text{Tr }\mathcal{E}(k)+\gamma(k)^{-1}|k|^{4(m+1)}\text{Tr }\mathcal{E}(k)).
	\end{aligned}
	\end{equation*}
	Now we choose sufficiently large $N_0>M$ such that for $N\geq N_0,$
	\begin{equation*}
		\mathcal{R}_t\leq \gamma_*t.
	\end{equation*}
	Define stopping times
	\begin{equation*}
		\sigma_n:=\inf\{t\geq 0:|\mathbf{u}_t^x|>n\},
	\end{equation*}
	then $\exp\{2M_{t\wedge\sigma_n}\}_{t\geq 0}$ is a martingales for each $n\geq 1$. Consequently, we obtain
	\begin{equation*}
	\begin{aligned}
		\mathbb{E}\exp\{8\tilde{C}\int_{0}^{t\wedge\sigma_n}h(\|Q_N\mathbf{u}_s^x\|)ds\}&\leq e^{8\tilde{C}}\mathbb{E}e^{\mathcal{M}_{t\wedge\sigma_n}+\mathcal{R}_{t\wedge\sigma_n}}\\
		&\leq e^{8\tilde{C}}\mathbb{E}e^{\mathcal{M}_{t\wedge\sigma_n}+\mathcal{R}_{t\wedge\sigma_n}}\\
		&\leq e^{8\tilde{C}}(\mathbb{E}e^{2\mathcal{M}_{t\wedge\sigma_n}})^{1/2}(\mathbb{E}e^{2\mathcal{R}_{t\wedge\sigma_n}})^{1/2}\\
		&\leq e^{\gamma_*t+8\tilde{C}},
	\end{aligned}
	\end{equation*}
	which completes the proof of (1).
	\par 
	(2) We can write (\ref{Equation L7}) in the form (\ref{Equation L0}) with $x$ changed to $y$, and with $W_t$  replaced by
	\begin{equation*}
	\tilde{W}_t:=W_t+\int_{0}^{t}\beta_sds,\quad\beta_t:= Q^{-1/2}(\tilde{\mathbf{u}}^y_t)P_N(\lambda(\mathbf{u}_t^x-\tilde{\mathbf{u}}_t^y)+B(\mathbf{u}_t^x,\mathbf{u}_t^x)-B(\tilde{\mathbf{u}}_t^y,\tilde{\mathbf{u}}_t^y)).
	\end{equation*}
	By $\mathbf{H_3}$, the inverse operator $Q^{-1/2}(\tilde{\mathbf{u}}^y_t)$ is well defined and bounded, thus there exists a constant $C_0>0$ such that for all $t\geq0$
	\begin{equation}\label{Equation L12}
	\begin{aligned}
	|\beta_t|&\leq C_0 (\lambda|P_N\mathbf{v}_t|+|B(\mathbf{u}_t^x,P_N\mathbf{v}_t)|+|B(\mathbf{v}_t,P_N\mathbf{u}_t^x-P_N\mathbf{v}_t)|)\\
	&\leq C_0(\lambda|P_N\mathbf{v}_t|+CN|\mathbf{u}_t^x|\,|P_N\mathbf{v}_t|+CN|P_N\mathbf{u}_t^x|\,|\mathbf{v}_t|+CN|\mathbf{v}_t|^2)\\
	&\leq C_0(\lambda|P_N\mathbf{v}_t|+2CN|\mathbf{u}_t^x|\,|\mathbf{v}_t|+CN|\mathbf{v}_t|^2).
	\end{aligned}
	\end{equation}

	Further, recall that for any $t>0$ the strong solution to equation (\ref{Equation L0}) with the initial value y, $\mathbf{u}^y_t$, is an image of the driving noise under some measurable mapping
	
	\begin{equation*}
	\Phi_t^y:C({0,t};\mathbb{R}^m)\rightarrow H.
	\end{equation*}
	In other words, $\mathbf{u}^y_t= \Phi_t^y (W_{[0,t]}),$ where $W_{[0,t]}:=\{W_s:s\in[0,t]\},\;t \geq 0.$ It follows from the Girsanov theorem that $\text{Law}(\widetilde{W}_{[0,t]})$ is absolutely continuous with respect to $\text{Law}(W_{[0,t]}).$ Therefore, by the uniqueness of the solution, we have $\tilde{\mathbf{u}}^y_t=\Phi_t^y (\widetilde{W}_{[0,t]}).$\par 
	Thus, by \cite[Theorem A.4]{BKS2020}, plugging into  (\ref{Equation L12}), we  derive that there exists $C$ such that for any $t\geq 0$
	\begin{equation*}
	\begin{aligned}
	d_{TV}(P_t(y,\cdot),\text{Law}(\tilde{\mathbf{u}}^y_t))&=d_{TV}(\text{Law}({\mathbf{u}}^y_t),\text{Law}(\tilde{\mathbf{u}}^y_t))\\
	&=d_{TV}(\text{Law}( \Phi_t^y (W_{[0,t]}),\text{Law}(\Phi_t^y (\widetilde{W}_{[0,t]})))\\
	&\leq d_{TV}(\text{Law}(W_{[0,t]}),\text{Law}(\widetilde{W}_{[0,t]}))\\
	&\leq C(\mathbb{E}(\int_{0}^{t}|\beta_s|^2ds)^{1/2}\\
	&\leq C(\sup\limits_{t\geq 0} \mathbb{E}|\mathbf{u}_t^x|^2+1)\int_{0}^{\infty}(\mathbb{E}|\mathbf{v}_t|^4)^{1/2}|dt.
	\end{aligned}
	\end{equation*}
	Applying Lemma \ref{Lemma L1} and (), we obtain that
	\begin{equation*}
		d_{TV}(P_t(y,\cdot),\text{Law}(\tilde{\mathbf{u}}^y_t))\leq C_2|x-y|^2,
	\end{equation*}
	where $C_2$ is a positive constant depending on $Q,\gamma_*,\mathcal{E}$, independent of $t$.
	\end{proof}

     \subsection{Proof of Lemma  \ref{HNS Lemma1}}

	\begin{proof}[Proof.] 
		(1) By \cite{KS2018}, formula (6.18), there exists $C_2>0$ depending only on $\nu$ and $D$ such that
		\begin{equation*}
		\limsup\limits_{t\rightarrow\infty}\frac{1}{t}\int_{0}^{t}(\|u\|^2+\|\partial_{z_2}u\|^2)ds\leq C_2(\|\sigma\|^2+\|\partial_{z_2}\sigma\|^2). 
		\end{equation*}
		Hence by (\ref{Equation HNS3}), $v(t)$ converges to zero almost surely, if $C\geq C_1C_2/\nu+L/\nu$.\par 
		
		(2) By the estimates in \cite{BKS2020}, one has
		\begin{equation*}
		d(|u|^2+|\partial_{z_2}u|^2)+2\nu(\|u\|^2+\|\partial_{z_2}u\|^2)dt=(\|\sigma\|^2+\|\partial_{z_2}\sigma\|^2)dt+M(t),
		\end{equation*}
		where $M$ is a continuous local martingale such that
		\begin{equation*}
		d\langle M\rangle_t\leq C_3(\|u\|^2+\|\partial_{z_2}u\|^2)dt,
		\end{equation*}
		for some $C_3>0$. \par 
		Using the same argument as in Lemma \ref{NS Lemma1}, let $\kappa>0$ and $	\Xi_\kappa:=\sup_{t\geq 0}(M_t-\kappa\langle M\rangle_t)$.	Then it follows 
		
		\begin{equation}\label{Equation HNS1}
		(2\nu-\kappa C_3)\int_{0}^{t}(\|u\|^2+\|\partial_{z_2}u\|^2)ds\leq (|x|^2+|\partial_{z_2}x|^2)+(\|\sigma\|^2+\|\partial_{z_2}\sigma\|^2)t+\Xi_{\kappa},
		\end{equation}
		and we fix $\kappa=\nu/ C_3$ for simplicity.\par

		Noting that the pseudo-inverse operator $\sigma^{-1}:H_N\rightarrow \mathbb{R}^m$ is well-defined and bounded, we can again employ the same proof as in Lemma \ref{NS Lemma1}, and deduce that for $\delta\in (0,1)$, there exists $C_\delta$ such that 
		\begin{equation*}
		d_{TV}(P_t(y,\cdot),\text{Law}(\tilde{u}(t)))\leq  C_\delta(\mathbb{E}(\int_{0}^{t}|v|^2ds)^\delta)^{1/(1+\delta)}.
		\end{equation*}

		It suffices to control $|v|^2$. For brevity, for $p>1$, denote 
		\begin{equation*}
		h_p(t)=(\nu\lambda_N-\frac{p+1}{2}L)t-C_1\int_{0}^{t}(\|u\|^2+\|\partial_{z_2}u\|^2)ds.
		\end{equation*}\par 
		Fix $\epsilon>0,$ by H\"{o}lder ineqaulity
		\begin{equation}\label{Equation HNS2}
		\begin{aligned}
		\mathbb{E}(\int_{0}^{t}|\mathbf{v}|^2ds)^\delta&=\mathbb{E}(\int_{0}^{t}|\mathbf{v}|^2\exp(h_p-\epsilon s)\exp(-h_p+\epsilon s)ds)^{\delta}\\
		&\leq \mathbb{E}[(\int_{0}^{t}|\mathbf{v}|^{2p}\exp[p(h_p-\epsilon s)]ds)^{\delta/p} (\int_{0}^{t}\exp[q(-h_p+\epsilon s)]ds)^{\delta/q}]\\
		&\leq (\int_{0}^{t}\mathbb{E}|\mathbf{v}|^{2p}\exp[p(h_p-\epsilon s)]ds)^{\delta/p}(\mathbb{E}(\int_{0}^{t}\exp[q(-h_p+\epsilon s)]ds)^{\frac{(p-1)\delta}{p-\delta}})^{\frac{p-\delta}{p}},
		\end{aligned}
		\end{equation}
		where $p>1,1/p+1/q=1.$	For $p>1,$ we use It\^o lemma again to obtain
		\begin{equation*}
		d|v|^{2p}=p|v|^{2p-2}d|v|^2+\frac{p(p-1)}{2}|\langle\sigma(u)-\sigma(\tilde{u}),v\rangle|^2dt,
		\end{equation*}
		hence  as in (\ref{Equation NS9.2}), we obtain that
		\begin{equation*}
		\int_{0}^{t}\mathbb{E}|v|^{2p}\exp[p(h_p-\epsilon s)]ds\leq (p\epsilon)^{-1}|x-y|^{2p},\quad\forall t\geq 0.
		\end{equation*}\par 
	
		To control the second term in (\ref{Equation HNS2}), we need some estimates of $\int_{0}^{t}(\|u\|^2+\|\partial_{z_2}u\|^2)ds.$
		By (\ref{Equation HNS1})
		\begin{equation*}
		-h_p(t)\leq \frac{ C_1}{\nu}(|x|^2+|\partial_{z_2}x|^2)+(-\nu\lambda_N+\frac{p+1}{2}L+\frac{ C_1}{\nu}B_0)t+\frac{ C_1}{\nu}\Xi_\kappa.
		\end{equation*}
		Denoted by $K:=\nu\lambda_N-L+\frac{ C_1}{\nu}B_0>0,\;k=\frac{C_1}{\nu}$, fixing $p=1+\frac{K}{L},\;\epsilon=\frac{K}{4}$, we have
		\begin{equation}\label{Equation HNS4}
		\begin{aligned}
		\mathbb{E}(\int_{0}^{t}\exp[q(-h_p(s)+\epsilon s)]ds)^{\frac{(p-1)\delta}{p-\delta}}&\leq \exp(\frac{kp\delta}{p-\delta}(|x|^2+|\partial_{z_2}x|^2))(\int_{0}^{t}e^{-\frac{(Kp\delta }{4(p-\delta)}s}ds)\\
             &\quad\cdot\mathbb{E}\exp(\frac{kp\delta}{p-\delta}\Xi_\kappa)\\
		&\leq \frac{Kp\delta }{4(p-\delta)}\exp(\frac{kp\delta}{p-\delta}(|x|^2+|\partial_{z_2}x|^2))	\mathbb{E}\exp(\frac{kp\delta}{p-\delta}\Xi_\kappa).
		\end{aligned}
		\end{equation}\par 
		Now we can choose $\delta$ such that $\frac{kp\delta}{p-\delta}=\kappa,$ then by (\ref{Equation HNS1}) and (\ref{Equation HNS4}),  we have
		\begin{equation*}
		\mathbb{E}(\int_{0}^{t}\exp[q(-h_p+\epsilon s)]ds)^{\frac{(p-1)\delta}{p-\delta}}\leq \frac{K\nu }{8kB_0}\exp(\frac{\nu}{4B_0}(|x|^2+|\partial_{z_2}x|^2)).
		\end{equation*} \par 
		Combining these results, we have
		\begin{equation*}
		\begin{aligned}
		d_{TV}(P_t(y,\cdot),\text{Law}(\tilde{u}_t))&\leq  C_\delta(\nu\lambda_N C_0)^{\frac{2\delta}{1+\delta}}(\mathbb{E}(\int_{0}^{t}|\mathbf{v}_s|^2ds)^\delta)^{1/(1+\delta)}\\
		&\leq C(\int_{0}^{t}\mathbb{E}|\mathbf{v}|^{2p}\exp[p(h_p-\epsilon s)]ds)^{\frac{\delta}{p(1+\delta)}}\\
            &\quad \cdot (\mathbb{E}(\int_{0}^{t}\exp[p(-h_p+\epsilon s)]ds)^{\frac{(p-1)\delta}{p-\delta}})^{\frac{p-\delta}{p(1+\delta)}}\\
		&\leq C|x-y|^{\frac{2\delta}{1+\delta}}\exp(\frac{\nu}{4B_0}(|x|^2+|\partial_{z_2}x|^2)),
		\end{aligned}
		\end{equation*}
		where $C>0,\;\delta\in(0,1),\;p>1$, are constants independent of $t$.
	\end{proof}

     \subsection{Proof of Proposition  \ref{HNS Prop UI}}

	\begin{proof}[Proof.] 
	We use the same arguments as in Proposition \ref{NS Prop UI}. Fix $x\in B(0,R)$ and $\epsilon\in(0,1)$, let $w$ be the solution of deterministic equation
	\begin{equation*}
	\left\{	\begin{aligned}
		&dw+(w\partial_xw+w\partial_zu+\partial_xp-\nu\Delta u)dt=0,\\
		&w_0=x+\frac{\epsilon}{2}e_1,
		\end{aligned}\right.
	\end{equation*}
	where $e_1$ is the eigenvector corresponding to eigenvalue $\lambda_1$.
	It is easy to show that 
	\begin{equation*}
	\int_0^t\|w\|^2ds\leq |x+\frac{\epsilon}{2}|^2\leq (R+1)^2 
	\end{equation*}
	Hence, we choose $T=T(R)>0$ sufficiently large such that $\|w_t\|<\epsilon/2$. \par 
	Define 
	\begin{equation*}
		D(t):=\|u_t-w_t\|^2-(\frac{\epsilon}{2})^2,
	\end{equation*}
	Then by similar estimasts as in \cite[Section 3.2]{NJG2017}, there exists constant $C>0$ such that
	\begin{equation*}
		dD(t)\leq-\nu\|u-w\|_{H^2}^2+C(\|u-w\|^4+\|u-w\|^3\|u-w\|_{H^2})dt+\|\sigma(u_t)\|^2dt+2\langle u-w, \sigma(u)\rangle_{H^1} dW,
	\end{equation*}
	and  $D(0)=0$. Define 
	\begin{equation*}
		\tau:=\inf\{t\geq 0:|D(t)|>(\frac{\epsilon}{4})^2\}.
	\end{equation*}
	For $t\in [0,\tau)$,
	\begin{equation*}
		\frac{3\epsilon^2}{16}<\|u_t-w_t\|^2< \frac{5\epsilon^2}{16}.
	\end{equation*}
    We can let $\epsilon$ small enough such that $C\|u-w\|^3<\nu$, hence  we obtain that for $t\in[0,\tau)$,
	\begin{equation*}
	    dD(t)\leq C\|u-w\|^4dt+B_0dt+2\langle u-w,\sigma(u)\rangle_{H^1}dW.
	\end{equation*}
	For the stochastic part, letting $M_t:=\int_{0}^{t}2\langle u_s-w_s, \sigma(u_s)\rangle_{H^1} dW_s$  be the local martingale part, we infer that
	\begin{equation*}
		\frac{d}{dt}\langle M\rangle_t=4\|\langle u_t-w_t,\sigma(u_t)\rangle_{H^1}\|^2\leq 4 B_0\|u_t-w_t\|^2,
	\end{equation*}
	 On ther other hand, by $\mathbf{H_2'}$, the inverse operator $\sigma^{-1}(u_t)$ is well defined and bounded, thus there exists a constant $C_0>0$ such that  $\|\sigma^{-1}(u_t)(y)\|\leq C_0\|y\|$ for $y\in V$. Consequently,
	\begin{equation*}
	\begin{aligned}
		\frac{d}{dt}\langle M\rangle_t&=4\|\langle u_t-w_t,\sigma(u_t)\rangle\|^2\geq 4 C_0\|u_t-w_t\|^2>\frac{3}{4}C_0\epsilon^2>0.
	\end{aligned}
	\end{equation*}\par 	
	Let now $\widetilde{W}$ be a Wiener	process that is independent of $W$ and define
	\begin{equation*}
		Y(t):=D(t\wedge \tau)+(\widetilde{W}_t-\widetilde{W}_\tau)\mathbf{1}_{\{t\geq\tau\}}.
	\end{equation*}
	Then $Y(t)$ is a semimartingale with $Y(0)=0$, which fulfills the conditions of \cite[Lemma I.8.3]{B1998} . Therefore, there exists $p>0$ such that for all $x\in B(0,R)$,
	\begin{equation*}
		\mathbb{P}(\sup\limits_{0\leq t\leq T}|Y(t)|\leq (\frac{\epsilon}{4})^2)\geq p.
	\end{equation*}\par 
	Then
	\begin{equation*}
		P_T(x,B(0,\epsilon))=\mathbb{P}(\|u_T\|<\epsilon)\geq \mathbb{P}(\tau> T)\geq 	\mathbb{P}(\sup\limits_{0\leq t\leq T}|Y(t)|\leq (\frac{\epsilon}{4})^2)\geq p
	\end{equation*}
	for all $x\in B(0,R),$ completing the proof.
	\end{proof}

	\subsection{Proof of Lemma  \ref{EV Lemma1}}

	\begin{proof}[Proof.] 
		(1) By \cite{KS2018}, formula (6.25), we obtain
		\begin{equation*}
		\limsup\limits_{t\rightarrow\infty}\frac{\nu}{t}\int_{0}^{t}\|\xi\|_{H^{-\gamma/2}}^2ds\leq B_0,
		\end{equation*}
		which implies (1)  if $C\geq C_1$, by (\ref{Equation HNS3}).
		
		(2) By the estimates in \cite{BKS2020}, one has
		\begin{equation*}
		d\|\xi\|_{H^{-\gamma/2}}^2+2\nu\|\xi\|_{H^{-\gamma/2}}^2dt=\|\rho(\mathbf{u})\|_{H^{-\gamma/2}}^2dt+M(t),
		\end{equation*}
		where $M$ is a continuous local martingale such that
		\begin{equation*}
		d\langle M\rangle_t\leq 4B_0\|\xi\|_{H^{-\gamma/2}}^2dt.
		\end{equation*} \par 
		Using the same argument as in Lemma \ref{NS Lemma1}, let $\kappa>0$ and $	\Xi_\kappa:=\sup_{t\geq 0}(M(t)-\kappa\langle M\rangle_t)$.	Then it follows 
		
		\begin{equation*}
		(2\nu-4\kappa B_0)\int_{0}^{t}\|\xi\|_{H^{-\gamma/2}}^2ds\leq \|\xi_0\|_{H^{-\gamma/2}}^2+\|\rho\|_{H^{-\gamma/2}}^2t+\Xi_{\kappa},
		\end{equation*}
		and we fix $\kappa=\nu/ (4 B_0)$ for simplicity.\par

		 We can write (\ref{Equation EV0}) in the form (\ref{Equation EV}) with $x$ changed to $y$, and with $W_t$  replaced by
		\begin{equation*}
		\widetilde{W}_t:=W_t+\int_{0}^{t}\beta_sds,\quad\text{with}\quad\beta_t:=\frac{1}{2}\nu\lambda_N^{\gamma/2-1/3}\sigma^{-1}P_N(\mathbf{u}_t-\tilde{\mathbf{u}}_t).
		\end{equation*}
		By our assumptions, the pseudo-inverse operator $\sigma^{-1}$ is well defined and bounded; thus there exists a constant $C_0>0$ such that for all $t\geq0$
		\begin{equation*}
		|\beta_t|\leq \nu\lambda_N^{\gamma/2-1/3}C_0 \|P_N\mathbf{v}_t\|_{L^2}.
		\end{equation*}
		By the Poincar\'{e} inequality, for any $t\geq 0$,
		\begin{equation*}
			\|P_N\mathbf{v}_t\|_{L^2}^2\leq \lambda_N^{\gamma/2}\|P_N\mathbf{v}_t\|_{H^{-\gamma/2}}^2\leq \lambda_N^{\gamma/2}\|\mathbf{v}_t\|_{H^{-\gamma/2}}^2.
		\end{equation*}\par 
		
		Hence we can again employ the same proof as in Lemma \ref{NS Lemma1}, and deduce that for $\delta\in (0,1)$, there exists $C_\delta$ such that 
		\begin{equation*}
		d_{TV}(P_t(y,\cdot),\text{Law}(\tilde{\mathbf{u}}(t)))\leq  C_\delta(\mathbb{E}(\int_{0}^{t}\|\mathbf{v}\|_{H^{-\gamma/2}}^2ds)^\delta)^{1/(1+\delta)}.
		\end{equation*}
		It suffices to control $\|\mathbf{v}\|_{H^{-\gamma/2}}^2$. Assume that $c:=C_1/C\in (0,1)$. It follows
		\begin{equation*}
		\begin{aligned}
		&\mathbb{E}(\int_{0}^{t}\|\mathbf{v}\|_{H^{-\gamma/2}}^2ds)^\delta\\
            &\leq \|x-y\|_{H^{-\gamma/2}}^{2\delta} \mathbb{E}(\int_{0}^{t}\exp(-\nu  s+C_1 \nu^{-1}\lambda_N^{-\gamma/2+1/3}\int_{0}^{t}\|\xi\|_{H^{-\gamma/2}}^2ds)^\delta\\
		&\leq \|x-y\|_{H^{-\gamma/2}}^{2\delta}\mathbb{E}(\int_{0}^{t}\exp(-\nu  s+c\nu\|\rho\|_{H^{-\gamma/2}}^{-2}(\|x\|_{H^{1-\gamma/2}}^2+\|\rho\|_{H^{-\gamma/2}}^2s+\Xi_{\kappa}))ds)^\delta\\
		&\leq \|x-y\|_{H^{-\gamma/2}}^{2\delta}\exp(\delta c\nu\|\rho\|_{H^{-\gamma/2}}^{-2}\|x\|_{H^{1-\gamma/2}}^2)(\nu(1-c))^{-\delta} \mathbb{E}\exp(\delta c\nu\|\rho\|_{H^{-\gamma/2}}^{-2}\Xi_\kappa).
		\end{aligned}
		\end{equation*}\par 		
		
		Therefore, taking $\delta\in (0,1)$ satisfying $\delta c\nu\|\rho\|_{H^{-\gamma/2}}^{-2}\leq \kappa$, and combining these results, we derive that
		\begin{equation*}
		\begin{aligned}
		d_{TV}(P_t(y,\cdot),\text{Law}(\tilde{u}(t)))&\leq  C_\delta(E(\int_{0}^{t}|v|^2ds)^\delta)^{1/(1+\delta)}\\
		&\leq \widetilde{C} \|x-y\|_{H^{-\gamma/2}}^{\frac{2\delta}{1+\delta}}\exp(\delta(1+\delta)^{-1} c\nu\|\rho\|_{H^{-\gamma/2}}^{-2}\|x\|_{H^{1-\gamma/2}}^2),
		\end{aligned}
		\end{equation*}
		where $\widetilde{C}>0$ is a constant independent of $t$.
	\end{proof}

	\subsection{Proof of Proposition \ref{EV Prop UI}}

	\begin{proof}[Proof.] 
	Fix $x\in B(0,R)$ and $\epsilon\in(0,1)$, let $\mathbf{w}$ be the solution of deterministic equation
	\begin{equation*}
	\left\{	\begin{aligned}
		&d\mathbf{w}+(\nu\mathbf{w}+\mathbf{w}_\gamma\cdot\nabla\mathbf{w}_\gamma dt+\nabla p )dt=0,\\
		&\mathbf{w}_0=x+\frac{\epsilon}{2}e_1,\quad \nabla\cdot\mathbf{w}=0,\quad\Lambda^\gamma\mathbf{w}_\gamma=\mathbf{w},
		\end{aligned}\right.
	\end{equation*}
	where $e_1$ is the  first eigenvector corresponding to $\Lambda^\gamma$.
	It is easy to show that 
	\begin{equation*}
		\|\mathbf{w}_t\|_{H^{-\gamma/2}}^2\leq e^{-2\nu t}|x+\frac{\epsilon}{2}e|^2 \leq e^{-2\nu t}(R+1)^2.
	\end{equation*}
	Hence, we choose $T=T(R)>0$ sufficiently large such that $\|\mathbf{w}_t\|_{H^{-\gamma/2}}^2<\epsilon/2$. 
	Define 
	\begin{equation*}
		D(t):=\|\mathbf{u}_t-\mathbf{w}_t\|_{H^{-\gamma/2}}^2-(\frac{\epsilon}{2})^2,
	\end{equation*}
	Then 
	\begin{equation*}
		dD(t)=-2\nu\|\mathbf{u}-\mathbf{w}\|_{H^{-\gamma/2}}^2dt+\|\sigma(\mathbf{u})\|_{H^{-\gamma/2}}^2dt+2\langle \mathbf{u}_t-\mathbf{w}_t, \sigma(\mathbf{u}_t)\rangle_{H^{-\gamma/2}}dW_t,
	\end{equation*}
	where $D(0)=0$. Define 
	\begin{equation*}
		\tau:=\inf\{t\geq 0:|D(t)|>(\frac{\epsilon}{4})^2\}.
	\end{equation*}
	For $t\in [0,\tau)$,
	\begin{equation*}
		\frac{3\epsilon^2}{16}<\|\mathbf{u}-\mathbf{w}\|_{H^{-\gamma/2}}^2< \frac{5\epsilon^2}{16}.
	\end{equation*}
	Hence we obtain that for $t\in[0,\tau)$,
	\begin{equation*}
	    -2\nu\|\mathbf{u}-\mathbf{w}\|_{H^{-\gamma/2}}^2dt+\|\sigma(\mathbf{u})\|_{H^{-\gamma/2}}^2dt\leq B_0 dt.
	\end{equation*}
	For the stochastic part, letting $M_t:=\int_{0}^{t}2\langle \mathbf{u}_s-\mathbf{w}_s, \sigma(\mathbf{u}_s)\rangle_{H^{-\gamma/2}}dW_s\rangle $  be the local martingale part, we infer that
	\begin{equation*}
		\frac{d}{dt}\langle M\rangle_t=4\langle \mathbf{u}_s-\mathbf{w}_s, \sigma(\mathbf{u}_s)\rangle_{H^{-\gamma/2}}^2\leq 4 B_0\|\mathbf{u}-\mathbf{w}\|_{H^{-\gamma/2}}^2,
	\end{equation*}
	 On ther other hand, by $\mathbf{EV_2'}$, the inverse operator $\sigma^{-1}(\mathbf{u}_t)$ is well defined and bounded, thus there exists a constant $C_0>0$ such that  $\|\sigma^{-1}(\mathbf{u}_t)(y)\|_{H^{-\gamma/2}}\leq C_0\|y\|_{H^{-\gamma/2}}$ for $y\in H^{1-\gamma/2}$. Consequently,
	\begin{equation*}
	\begin{aligned}
		\frac{d}{dt}\langle M\rangle_t&=4\langle \mathbf{u}_t-\mathbf{w}_t, \sigma(\mathbf{u}_t)\rangle_{H^{-\gamma/2}}^2\geq 4 C_0\|\mathbf{u}_t-\mathbf{w}_t\|_{H^{-\gamma/2}}^2>\frac{3}{4}C_0\epsilon^2>0.
	\end{aligned}
	\end{equation*}\par 	
	Let now $\widetilde{W}$ be a Wiener	process that is independent of $W$ and define
	\begin{equation*}
		Y(t):=D(t\wedge \tau)+(\widetilde{W}_t-\widetilde{W}_\tau)\mathbf{1}_{\{t\geq\tau\}}.
	\end{equation*}
	Then $Y(t)$ is a semimartingale with $Y(0)=0$, which fulfills the conditions of \cite[Lemma I.8.3]{B1998} . Therefore, there exists $p>0$ such that for all $x\in B(0,R)$,
	\begin{equation*}
		\mathbb{P}(\sup\limits_{0\leq t\leq T}|Y(t)|\leq (\frac{\epsilon}{4})^2)\geq p.
	\end{equation*}\par 
	Then
	\begin{equation*}
		P_T(x,B_d(0,\epsilon))=\mathbb{P}(\|\mathbf{u}_T\|_{H^{-\gamma/2}}<\epsilon)\geq \mathbb{P}(\tau> T)\geq 	\mathbb{P}(\sup\limits_{0\leq t\leq T}|Y(t)|\leq (\frac{\epsilon}{4})^2)\geq p
	\end{equation*}
	for all $x\in B(0,R),$ completing the proof.
	\end{proof}

    \vspace{0.6em}

    \noindent {\bf Acknowledgments.}  I am grateful to my supervisor, Professor Yong Liu, for drawing my attention to these problems and for many insightful discussions. I would like to thank Professor Fuzhou Gong and Professor Yuan Liu for their valuable comments and suggestions. I am thankful to the thesis defense committees and anonymous reviewers for their kind encouragement and helpful feedback.

    \bibliographystyle{abbrv}
    \bibliography{reference}

\end{document}